\documentclass[reqno,12pt,letterpaper]{amsart}
\usepackage{amsmath,amssymb,amsthm,graphicx,mathrsfs,url}
\usepackage[usenames,dvipsnames]{color}
\usepackage[colorlinks=true,linkcolor=Red,citecolor=Green]{hyperref}
\usepackage{amsxtra,esint}
\usepackage{mathtools,wasysym}
\mathtoolsset{showonlyrefs,showmanualtags}

\usepackage{aliascnt}
\usepackage{float}
\usepackage[font=small,labelfont=bf]{caption}
\usepackage{hyperref,cleveref,amssymb,mathtools}
\usepackage{dsfont}

\usepackage[font=small,labelfont=bf]{caption}
\usepackage[style=alphabetic, backend=biber,doi=false,isbn=false,url=false,eprint=false,giveninits=true,maxbibnames=99]{biblatex}
\def\?[#1]{\textbf{[#1]}\marginpar{\Large{\textbf{??}}}}

\def\smallsection#1{\smallskip\noindent\textbf{#1}.}
\newtheorem*{theorem*}{Theorem}
\newtheorem{theorem}{Theorem}
\newtheorem{prop}{Proposition}[section]
\theoremstyle{definition}
\newtheorem{defi}{Definition}[section]
\newtheorem{conj}{Conjecture}
\newtheorem{lemma}[prop]{Lemma}
\newaliascnt{corr}{prop}
\newtheorem{corr}[corr]{Corollary}
\aliascntresetthe{corr}
\newtheorem{rem}{Remark}[section]
\newtheorem{ex}{Example}[section]

\numberwithin{equation}{section}

\newcommand{\abs}[1]{\left | #1 \right| }
\DeclareMathOperator{\tr}{Tr}

\newcommand*{\dd}{\mathop{}\!\mathrm{d}}
\DeclareMathOperator{\Tr}{Tr}
\DeclareMathOperator{\Spec}{Spec}

\DeclareMathOperator{\supp}{supp}
\DeclareMathOperator{\dist}{dist}
\DeclareMathOperator{\Span}{Span}

\DeclareMathOperator{\Res}{Res}
\newcommand{\ip}[2]{\left  \langle#1,#2 \right \rangle}
\newcommand{\mat}[1]{\begin{pmatrix} #1 \end{pmatrix}}

\newcommand{\set}[1]{ \left \{ #1 \right  \}}
\newcommand{\N}{\mathbb{N}}
\newcommand{\p}{\partial}
\newcommand{\dbar}{\overline{\p}}
\newcommand{\Z}{\mathbb{Z}}

\newcommand{\e}{\varepsilon}
\renewcommand{\epsilon}{\varepsilon}

\newcommand{\R}{\mathbb{R}}

\renewcommand{\SS}{\mathbb{S}}
\newcommand{\norm}[1]{ \left \| #1 \right  \|  }

\newcommand{\C}{\mathbb{C}}
\renewcommand{\phi}{\varphi}

\newcommand{\1}{\mathds{1}}

\renewcommand{\Re}[1]{{\rm{Re}} \left ( #1\right ) }
\renewcommand{\Im}[1]{{\rm{Im}} \left ( #1 \right ) }

\newcounter{step}

\newcommand{\Step}{
    \refstepcounter{step} 
    \noindent \textbf{Step \thestep.} 
}

\title[Poisson formula for SdS]{A Poisson formula for the wave propagator on  Schwarzschild-de Sitter backgrounds}
\author{Izak Oltman} 
\address[Izak Oltman]{Department of Mathematics, Northwestern University, 2033 Sheridan Rd, Evanston, IL 60208}
\email{ioltman@northwestern.edu}
\author{Ben Pineau*}
\address[Ben Pineau*]{Courant Institute for Mathematical Sciences\\
New York University
} \email{brp305@nyu.edu}

\begin{document}

\begin{abstract}
This paper proposes a Poisson formula for the wave propagator of the Schwarzschild--de Sitter (SdS) metric. 
That is done by proving a Poisson formula relating wave propagators and scattering resonances for a class of non-compactly supported potentials on the real line. 
That class includes the Regge--Wheeler potentials obtained from separation of variables for SdS. 
The novelty lies in allowing non-compact support -- all exact Poisson formulae of Lax--Phillips, Melrose, and other authors required compactness of the support of the perturbation. 

A key feature of the analysis is the presence of an exceptional class of potentials for which outgoing solutions may vanish at certain non-resonant frequencies. We identify and describe this class, which we call the resonant condition.
\end{abstract}

\maketitle

\section{Introduction}
{Quasinormal modes of black holes are supposed to measure the ringdown of gravitational waves -- see \cite{berti2009quasinormal} for a survey, \cite{cardoso2018quasinormal,jaramillo2021pseudospectrum,yang2025black} for more recent work and references. 
Their study has been of interest to mathematicians since the work of Bachelot--Motet-Bachelot \cite{BMB} -- see Hintz--Xie \cite{hintz2021quasinormal},
Hitrik--Zworski \cite{hitrik2024overdamped}, and J\'ez\'equel \cite{jezequel2022upper} for some recent advances. }

{The quasinormal modes\footnote{We follow the informal convention, common in math and physics, of referring to both the frequencies and the corresponding states as modes; the former should be more correctly called quasinormal frequencies.} can be considered as scattering resonances for linear wave equations arising in general relativity -- see \cite[Chapter 5]{dyatlov2019mathematical} and references given there. One of the features of the theory of scattering resonances, going back to the work of Lax--Phillips, is Poisson formulae. They generalized the Poisson formula for the wave equation on compact Riemannian manifolds $ ( M, g ) $, 
\begin{equation}
\label{eq:Poisson_comp}
\tr \cos t \sqrt { - \Delta_g } = \tfrac12 \sum_{ \lambda_j^2 \in \Spec ( - \Delta_g ) }
e^{ i \lambda_j t } .
\end{equation}
Here $ - \Delta_g $ is the Laplace--Beltrami operator for the metric $ g $ and 
the traces are meant in the sense of distributions in $ t $. The formula \eqref{eq:Poisson_comp} is, of course, an immediate consequence of the spectral decomposition of $ - \Delta_g $.
In the non-compact setting, the left-hand side
is not distributionally of trace class and it has to be renormalized. 
The right-hand side should then be replaced by a sum over scattering resonances with terms $ e^{ - i \lambda_j |t| } $ (since $ \Im {\lambda_j } < 0$) and the formula is typically not valid at $ t = 0 $ -- see \cite[\S\S 2.6, 3.10, 7.4]{dyatlov2019mathematical} for an introduction to this subject and references. }

{To investigate a possible Poisson formula for quasinormal modes of black holes, that is, a relation between (renormalized) traces of wave groups and sums over the modes, we consider the simplest case of Schwarzschild--de Sitter black holes. 
There, a standard Regge--Wheeler reduction gives quasinormal modes as the union of scattering resonances of a family of potentials parametrized by angular momenta. 
}

Motivated by this, we compute here the trace of the propagator for the wave equation $\p_t^2 - \p_x^2 + V(x)$ for $ V $ which include the Regge--Wheeler potentials and their perturbations -- see \S \ref{s:RW}. 

\subsection{Trace formula for a class of non-compactly supported potentials.}
\label{s:trace}
We consider potentials satisfying the following hypothesis.

There exist  constants $ A_\pm > 0 $ and $ R>0$ such that
\begin{equation}
    \label{hyp:1}
V ( x ) |_{ \pm  x > R } = F_\pm (\exp (  \mp  x A_\pm ) ),  
\end{equation}
where $  z \mapsto F_\pm ( z ) $ are holomorphic in a neighborhood of $ z = 0$, $F_\pm (0)= 0$,  {and $F_\pm '(0) \neq 0$}.

In particular, this means there exist $\set{V^\pm_j \in \R : j\in \Z_{\ge 1}}$ and constants $A,C>0$ so that
\begin{equation}
    V(x)|_{ \pm x > C } = \sum_{j=1}^\infty \left(e^{\mp xA_\pm} \right)^jV^\pm_j  \label{eq:169} 
    \end{equation}
with $\abs{V^\pm _j } \le A^j$  {and $V_1^\pm \neq 0$}.

 {
We must additionally define a subset of these potentials which we say satisfy the \textit{Resonant Condition}, denoted by $\mathcal{RC}$.
This class of potentials should be thought of as a ``measure-zero'' edge case.
Further discussion will be postponed until \Cref{ss:res_cond}.

\begin{defi}[Minimal resonant condition definition]\label{def:min_res_cond}
For $j\in \Z_{\ge 2}$, we say $V\in \mathcal{RC}_j^\pm$ if $V_j^\pm$ (the coefficient in \eqref{eq:169}) is equal to a specific polynomial depending on $V_1^\pm,\dots V_{j-1}^\pm$, and $A_\pm$ (see \eqref{eq:421}).
We denote $\mathcal{RC} = \bigcup_\pm \bigcup _j \mathcal{RC}_j^\pm$
\end{defi}
}

The first result, which is implicit in \cite[\S 3]{dyatlov2011quasi} and earlier works, is the meromorphy of the resolvent of the corresponding Schr\"odinger operator (Green function):
\begin{prop}
\label{p:0}
The resolvent of $ P_V \coloneq D_x^2 + V ( x ) $, $D_x \coloneq \p_x /i ,$
\[   R_V ( \lambda ) \coloneq  ( P_V - \lambda^2 )^{-1} \colon L^2 \to L^2, \ \ \  \Im 
\lambda  \gg 1 , \]
continues meromorphically to 
\[ R_V ( \lambda ) \colon L^2_{\rm{c}}  ( \mathbb R) \to L^2_{\rm{loc} } ( \mathbb R), \ \ \ \lambda \in \mathbb C.\]
\end{prop}
In other words, for any $ f, g \in L^2_{\rm{c}} $, 
$ \lambda \mapsto \langle R_V ( \lambda ) f , g \rangle $ is a 
meromorphic function with poles independent of $ f $ and $ g $. 
The poles are called {\em scattering resonances} and we denote their set (with elements included according to their multiplicity -- see \Cref{def:multi}) $ {\rm{Res}} (V) $.
For recent interesting developments in 1D scattering, see the work of Korotyaev \cite{Ko04,Ko05,Ko14}.

Define $\Box_V \coloneq \p_t^2 + P_V $ and let $U_V(t)\coloneq \cos (t\sqrt {P_V})$ be the cosine wave propagator with respect to $\Box_V$.
That is, if $u\in C_0^\infty (\R)$:
\begin{align}\label{eq:205}
    \begin{cases}
        \Box_V (U(t)u(x))(t,x) = 0, \\
        U(0)u(x) = u(x),\\
        (\p_t (U(t) u(x) )\big|_{t=0} = 0.
    \end{cases}
\end{align}
Our main result relates the trace of the wave propagator minus the free wave propagator to a sum over resonances, analogous to \eqref{eq:Poisson_comp}. 

\begin{theorem*}[Main result]
Suppose $V\in C^\infty (\R ;\R)$ satisfies \eqref{hyp:1}, $V \notin \mathcal{RC}$ (see \Cref{def:min_res_cond}) and $U_V(t)$ is the wave propagator for $\Box_V$.
Then for $t >0$ and in the sense of distributions,
\begin{align}
\Tr(U_V(t) - U_0(t))  = \tfrac{1}{2}\sum_{\lambda_j \in {\Res(V)}} e^{-i\lambda_j t} + \mathcal{A}_+ ( t ) + \mathcal{A}_ -( t )  - \tfrac 1 2, \label{eq:218}
\end{align}
where 
\begin{align}
 \mathcal A_\pm ( t )  = - \tfrac12 ( e^{t A_\pm/2 } - 1 )^{-1} \label{eq:209}
 \end{align}
\end{theorem*}

In \eqref{eq:218}, and throughout the paper, sums over resonances are taken with multiplicity (see \Cref{def:multi}).
A more general theorem for $V\in \mathcal{RC}$ is presented in \Cref{ss:res_cond}.

\begin{ex}[P\"oschl-Teller potential]
The scattering matrix, $S(\lambda)$, for the P\"oschl-Teller potential
\begin{align}
    V_{\rm{PT}}(x) \coloneq \frac{\ell^2 + 1/4}{\cosh^2 (x)}\label{eq:PTpotential}
\end{align}
can be explicitly computed \cite{cevik2016resonances} and the resonances are given by $$\set{\lambda_j^\pm = -i(j+1/2)\pm \ell : j\in \Z_{\ge 0}}.$$
Moreover, for $\ell \in \R$, $V_{\rm{PT}}(x)\notin\mathcal{RC}$ (see \Cref{prop:rc3}).
The Birman--Kre\u{\i}n trace formula can be applied (exactly as will be done in \S \ref{s:birman}, but with the poles of $\frac{d}{d\lambda} \log \det S(\lambda)$ explicit).
It can then be shown that
\begin{align}
    \Tr (U_{V_{\rm{PT}}}(t) - U_0(t)) &= \tfrac 1 2 \sum_{\lambda_j \in \Res( V_{\rm{PT}}) } e^{-i\lambda_j   t }  \underbrace{- \frac{1}{e^{t}-1}}_{\mathcal{A_+}+\mathcal A_-} - \tfrac 1 2 \\
    &=  \frac{\cos(\ell t)}{2\sinh(t/2)}- \frac{e^{-t/2}}{2\sinh(t/2)}- \tfrac 1 2\\
    &=\tfrac{1}{2} \left(\frac{\cos(\ell t)-e^{-t/2}}{\sinh(t/2)} - 1\right), \ \ \  t > 0 . \label{eq:243}
\end{align}
We can compare this to a numerically computed approximation of the flat trace of $U_{V_{\rm{PT}}}(t)$, which is presented in \Cref{fig:theonlyfigure}.
\begin{figure}
    \centering
    \includegraphics[scale=.35]{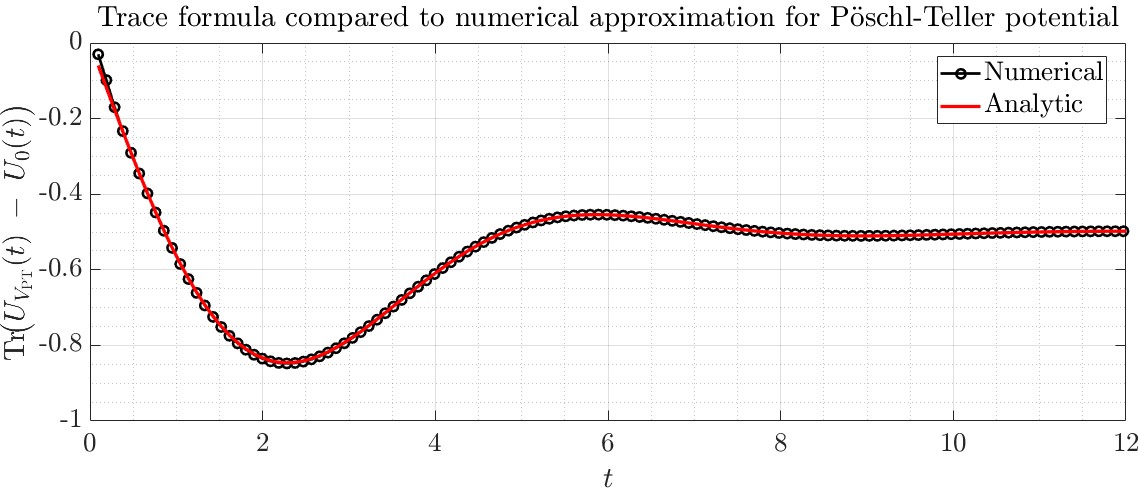}
    \caption{Comparison between a numerical approximation of the flat trace for the Pöschl–Teller potential (with $\ell = 1$) and the analytic expression in \eqref{eq:243}. 
    The numerical trace is computed in MATLAB by truncating and discretizing the spatial domain, evolving a normalized Gaussian centered at each grid point, $x_i$, under the cosine propagator $\cos\big(t\sqrt{P_{V_{\mathrm{PT}}}}\big)$, subtracting the free wave evolution of the same Gaussian, and evaluating at $x_i$.
    This approximates $U_{V_{\rm{PT}}}(t,x_i,x_i)-U_{0}(t,x_i,x_i)$ (the Schwartz kernels).
    Computing the sum of these terms, multiplied by $x_{i} - x_{i-1}$ will approximate the flat trace.}
    \label{fig:theonlyfigure}
\end{figure}{
Note that as $t\to 0^+$, the right-hand side of \eqref{eq:243} goes to zero.
This is a classical result for compactly supported potentials (see Smith \cite{smith2023trace} and references given for a recent account).

A similar explicit example (but for a potential satisfying the \textit{Resonant condition}) is presented in \Cref{prop:581}.
}
\end{ex}

{When $ V$ is compactly supported, Proposition \ref{p:0} is classical -- see \cite[\S 2.2]{dyatlov2019mathematical}, and we have the Poisson formula 
\begin{equation} 
\label{eq:Melr} \Tr(U_V(t) - U_0(t))  = \tfrac{1}{2} {\rm{p.v.} } \sum_{\lambda_j \in {\rm{Res}}(V)} e^{-i\lambda_j |t|} - |{\rm{chsupp}} V | \delta_0 ( t ) - \tfrac12 , \end{equation}
see \cite[Theorem 2.21]{dyatlov2019mathematical}.
Both sides are defined distributionally on $ \mathbb R  $ and $ \rm{p.v.} $ denotes the distributional principal value obtained by summing over $ |\lambda_j | \leq R $ and taking the $ R \to \infty $ limit. 
The notation $ \rm{ch} $ denotes
the convex hull. 
The presence of the $|{\rm{chsupp}} V | \delta_0 ( t ) $ term explains the need for the renormalization terms $ \mathcal{A}_\pm $ in the non-compact case. 
 For $ t > 0 $, \eqref{eq:Melr} was proved by Melrose \cite{melrose1982scattering} and for $ t > 2 {\rm{diam}}(\supp V ) $ by Lax--Phillips \cite{LaxPhillips1978}.}
The proof of such trace identities often relies on the Eisenbud–-Wigner delay operator, $\partial_\lambda S(\lambda) S(\lambda)^{-1}$. 
In the compactly supported setting, this operator is linked to the Breit-–Wigner series $\sum_{\lambda_j \in  \Res(V)\setminus{0}} \Im {\lambda_j} |\lambda_j|^{-2}$ via the Breit-–Wigner approximation (cf. \cite[Theorem 2.20]{dyatlov2019mathematical}). 
For non-compactly supported potentials, however, the convergence or divergence of this approximation remains an open question; see Backus \cite{backus2020breit} for a thorough discussion.

If we define $\iota \colon (t,x)\mapsto (t,x,x)$, $\pi\colon (t,x)\mapsto t$, and denote the kernel of $U_V(t)$ by $U_V (t,x,y)$, as a distribution on $(0,\infty)\times\R^2$, {with the pairing using the Lebesgue measure}, then the \textit{flat trace} is defined as
\begin{align}
    \tr^\flat U_V(t) \coloneq \pi _* \iota^*U_V \label{eq:flat_trace}
\end{align}
(which is well-defined thanks to the usual wave front set \cite[\S 29.1]{hormander4}).
Because $\tr^\flat U_0(t)=0$ and $\tr^\flat (U_V - U_0) = \tr(U_V-U_0)$, \eqref{eq:218} can be written
\begin{align}
    \Tr^\flat U_V (t)  = \tfrac{1}{2}\sum_{\lambda_j \in {\rm{Res}}(V)} e^{-i\lambda_j t} + \mathcal{A}_+ ( t ) + \mathcal{A}_ -( t ) -\tfrac{1}{2}.\label{eq:231}
\end{align}

\subsection{Regge--Wheeler potentials}

As a consequence of our main theorem, we propose a trace formula for the wave propagator in the Schwarzschild-de Sitter metric. 
It necessarily involves a renormalization of the trace -- not the flat trace -- which can be used to pass from a trace formula for each Regge--Wheeler potential (parameterized by angular momenta), to a trace formula for the full potential. 
Because $\mathcal{A}_{\pm}(t)$ are independent of angular momentum, simply summing over angular momentum is too naive.
On the other hand, thanks to recent work of 
J\'ez\'equel \cite{jezequel2022upper} the sum of $e^{-i\lambda_j t}$ over all quasinormal modes (that is, the sum over $ {\rm{Res}} ( V_\ell ) $) is a well-defined distribution.

\label{s:RW}
Schwarzschild-de Sitter spacetime can be constructed in the following way (c.f. \cite[\S 3.1]{jezequel2022upper}). 
Define $M \coloneqq  (r_- , r_+) _r \times \SS^2$
where $0 < r_- < r_+$ are the positive roots of
\begin{align}
    G(r) \coloneqq 1 - \frac{2m}{r} - \frac{\Lambda r^2 }{3}  .\label{eq:120}
\end{align}
Here $m >0$ is the \textit{black hole mass} and $\Lambda \in (0,(9m^2)^{-1})$ is the \textit{cosmological constant}. 

Define the Schwarzschild-de Sitter spacetime as $\hat M \coloneqq \R_t \times M$ and let $g$ be the Lorentzian metric:
\begin{align}
    g \coloneqq -G \dd t^2 + G^{-1} \dd r^2 + r^2  g_\SS 
\end{align}
with $g_\SS$ the standard metric on $\SS^2$.

With this Lorentzian metric, we have the wave operator $\Box_g $, which satisfies, for $u\in C_0^\infty((r_- , r_+) \times \SS^2)$
\begin{align}
    \Box_g u = - \frac{1}{G} D_t^2 u + r^{-2} D_r (r^2 G D_r u ) + r^{-2} \Delta _{\SS^2} u 
\end{align}
(cf. \cite[Chapter 4]{chandrasekhar1998mathematical}).
Let $U_g(t)$ be the Dirichlet wave propagator with respect to $\Box_g$ (i.e., $U_g(t)$ satisfies \eqref{eq:205} with $\R$ replaced by $M$ and $\Box_V$ replaced by $\Box_g$).

Following \cite[\S 4]{barreto1997distribution}, write:
\begin{align}
    \Box_g = -G^{-1} D_t^2 + G^{-1} P \label{eq:boxg}
\end{align}
where
\begin{align}
    P \coloneqq G r^{-2} D_r (r^2 G) D_r + G r^{-2} \Delta _{\SS ^2} .
\end{align} 
Let $\tilde P \coloneq r Pr^{-1}$ and define $x= x(r)$ such that $x'(r) = G^{-1}$ so that
\begin{align}
    \tilde P = D_x^2 + G r^{-2}\Delta _{\mathbb{S}^2} + G r^{-1} (\p_r G). \label{eq:161}
\end{align}
Decomposing a function $u(x,\omega)$ by spherical harmonics: $u(x,\omega) = \sum_{\ell,k} a_{\ell,k} (x) Y_\ell^k (\omega) $ (where $\Delta _\omega Y_\ell^k(\omega) = \ell (\ell +1 ) Y_\ell^k (\omega)$) we get that $u$ is a generalized eigenfunction of $\tilde P$ with eigenvalue $\lambda^2$ if and only if:
\begin{align}
    (D_x^2 +  \underbrace{G r^{-2}  ( \ell (\ell +1 ) +r  \p_r G) }_{\coloneqq V_\ell (x)}    ) a_{\ell,k} = \lambda^2 a_{\ell,k}. \label{eq:165}
\end{align}
Note for each $\ell$, $V_\ell(x(r))$ is a smooth function on the interval $r\in (r_-,r_+)$ vanishing at $r= r_\pm$.

Now, if $U_{\tilde V}(t)$ is the propagator for $-D_t^2 + \tilde P$ and $U_{V_\ell}(t)$ is the propagator for $-D_t^2 + D_x^2 +  V_\ell (x)$, then for $u_0\in C^\infty _0(\R_x\times \mathbb S ^2 _\omega) $
\begin{align}
    U_{\tilde V}(t) u_0 = \sum_{\ell=0}^\infty \sum_{k = -\ell}^\ell    Y_\ell ^k  U_{V_\ell}\ip{u_0}{Y^k_\ell}_{L^2 _\omega (\mathbb S ^2 )}. \label{eq:179}
\end{align}
Here $\ip{u_0}{Y^k_\ell}_{L^2 _\omega (\mathbb S ^2 )}$ is integration of $u_0$ against $Y^k_\ell$ in the $\omega$ variable, so the resulting object is a function of $x$ alone.

Using this, we can rewrite the propagator for $\Box_g$ (defined in \eqref{eq:boxg}) as
\begin{align}
[U_g  ( t )  F] ( r , \omega ) = r^{-1} [ U_{\tilde V } (t)  ( r F )]  ( r, \omega ) , \ \ 
F = F ( r,\omega ) . \label{eq:185}
\end{align}

It is straightforward to verify that $V_\ell (x)$ satisfies \eqref{hyp:1} with $A_\pm = \abs{\p_r G(r_\pm)}$.
Indeed, by fixing $r_0 \in (r_- ,  r_+)$ we can define
\begin{align}
    x(r) = \int_{r_0}^r (G(s))^{-1} \dd s
\end{align}
so that $x(r_\pm ) =  \pm \infty $. Near $r = r_+$, we have:
\begin{align}
    G(r) = (r_+ - r) f(r) 
\end{align}
where $f(r)$ is nonzero and holomorphic near $r = r_+$. 
Taylor expanding $f$ and integrating, we get for $r$ near $r_+$
\begin{align}
    x(r) = -A_+^{-1}\log(r_+ - r) +b(r)\label{eq:182}
\end{align}
for $b$ holomorphic near $r_+$. 
Therefore $w\coloneq \exp(-A_+ x)=(r-r_+)\exp(b(r))$.
By the inverse function theorem, we get (for $w$ near $0$) $V_\ell (x(w))=F_+ (w)$ with $F$ holomorphic near $w= 0$.
We similarly get, for the change of coordinates $w = \exp(A_- x)$, $V_\ell (x(w)) = F_-(w)$ is holomorphic near $w = 0$.

\begin{theorem*}[Trace formula for each spherical harmonic]
    For each $\ell \in \Z_{\ge 0}$, $t >0$
\begin{align}
            \Tr_{L^2 ( \mathbb R )} (U_{V_\ell} (t) - U_0(t)) & = \tfrac{1}{2}\sum_{\lambda_j \in \Res(V_\ell)} e^{-i\lambda_j t} + \mathcal A_{+,\ell} (t) + \mathcal{A_{-,\ell} }(t) - \tfrac{1}{2} \label{eq:369}
\end{align}
    in the sense of distributions, where $U_{V_\ell}$ is the propagator for $V_\ell$ (defined in \eqref{eq:165}), $\Res(V_\ell)$ is the set of resonances of $V_\ell$,
    \begin{align}
        \mathcal{A}_{\pm,\ell} (t) = -\tfrac 1 2  \sum_{n \in \Z_{\ge 1 } \setminus \Lambda_{\pm,\ell}} e^{-tA_\pm n / 2}
    \end{align}
    where $\Lambda_{\pm,\ell} \subset \Z_{\ge 2}$ are defined in \Cref{def:rc}, and $A_\pm = \abs{\p_r (G(r_\pm))}$. 
    \end{theorem*}

    \begin{rem}
        As we will see in \Cref{ex:RW in RC}, $\Lambda_{\pm,\ell}$ can be nonempty for certain $\ell, m$, and $\Lambda$.
        We however conjecture that for each $\Lambda$, outside of a measure zero set of $m$, $\Lambda_{\pm,\ell}  = \emptyset$.
    \end{rem}

\begin{rem}
The constants $A_\pm$ are related to the surface gravity of the event and cosmological horizon, $\kappa_-$ and $\kappa_+$, respectively,  by the simple relation $A_\pm = 2 \kappa_\pm$ (cf. \cite[\S 12.5]{wald2024general}).
When $\Lambda_\pm = \emptyset$, we can then write the terms in \eqref{eq:369} as
\begin{align}
\mathcal A_{\pm ,\ell }(t) =  -\tfrac 1 2 \sum_{m=0}^\infty e^{-i(-i m \kappa_\pm )t}.
\end{align}
This suggests that these terms can be included in the sum over resonances.
Interestingly, on the de Sitter space, there are resonances at $\set{-im\kappa_\pm : m \in \Z_{\ge 0}}$ (cf. \cite{hintz2021quasinormal}).
Making sense of this connection could lead to a formulation of the Poisson formula for other metrics.
\end{rem}

The trace on the left-hand side of \eqref{eq:369} is written in the $x$ variable.
Changing variables $x$ to $r$, the left-hand side of \eqref{eq:369} is
\begin{align}
    \Tr_{L^2 ((r_- ,r_+), G(r)^{-1} \dd r)} (U_{V_\ell }(t) - U_0(t)). \label{eq:392}
\end{align}
Recalling the notion of flat trace \eqref{eq:flat_trace}, \eqref{eq:392} can be written
\begin{align}
     \Tr^\flat_{L^2 ((r_- ,r_+), G(r)^{-1} \dd r)} (U_{V_\ell }(t) ). \label{eq:396}
\end{align}
Conjugating $U_{V_\ell}$ by $r$ allows us to write \eqref{eq:396} as
\begin{align}\label{eq:409}
    \Tr^\flat_{L^2 ((r_- ,r_+), r^2G(r)^{-1} \dd r)} (r^{-1}U_{V_\ell }(t) r) .
\end{align}
Let $\Pi_\ell$ be the orthogonal projection in the $\omega$ variable onto the space spanned by $\set{Y_\ell ^k (\omega) : k = -\ell,- \ell + 1,\dots ,\ell}$.
Then by \eqref{eq:179} and \eqref{eq:185}, we can rewrite \eqref{eq:409} as
\begin{align}
\Tr^\flat_{L^2 ((r_-,r_+)\times \mathbb S^2  ,r^2 G^{-1}\dd r \dd \omega)} ( \Pi _\ell  U_g (t) \Pi_\ell ).
\end{align}
So if we now define the \textit{gravitationally renormalized trace} on $$\text{SdS}\coloneq L^2((r_- ,r_+)\times \mathbb S ^2 , G(r)^{-1}r^2 \dd r \dd \omega) $$ as
\begin{align}
    \widetilde \Tr ^\flat_{\rm{SdS}} (A(t)) \coloneq \sum_{\ell =0}^\infty\left (  \Tr^\flat_{\rm{SdS}}\Pi _\ell A(t) \Pi_\ell -\mathcal{A}_{+,\ell} (t) - \mathcal{A}_{-,\ell} (t)   +\tfrac{1}{2}\right )
\end{align}
then \eqref{eq:369} gives rise to the following global trace formula.
\begin{theorem}[Global trace formula]
The gravitationally renormalized trace of the wave propagator on the Schwarzschild-de Sitter metric is
\begin{align}
    \widetilde \Tr ^\flat_{\rm{SdS}} (U_g(t)) = \tfrac 1 2 \sum_{\lambda \in \Res_g} e^{-it\lambda}, \quad t> 0 
\end{align}
where $\Res_g \coloneq \bigcup_{\ell} \Res(V_\ell)$ (and the term on the right-hand side is a well-defined distribution by \cite[Theorem 3]{jezequel2022upper}).
\end{theorem}

 {
\subsection{Resonant Condition}\label{ss:res_cond}

Our analysis of the resonances involves the construction of incoming and outgoing solutions to $-\p_x^2 +V(x)$.
In certain pathological cases of potentials, one of these solutions can identically vanish for a frequency $\lambda$ when $\lambda$ is not a resonance.
Generically (see \Cref{prop:rc1}) this does not happen.
However, the existence of such pathological examples, requires additional definitions and hypotheses on our potential $V(x)$.

\begin{defi}[Resonant condition]\label{def:rc}
For $j\in \Z _{\ge 2}$, denote $\mathcal{RC}_j^\pm= \mathcal{RC}_j^\pm(A_\pm)$ (for \textit{resonant condition}) the set of potentials $V(x)$ satisfying \eqref{eq:169} (with $V_1^\pm \neq 0$) such that
\begin{align}
    V_j^\pm=\sum_{r=2}^j (-1)^r A_\pm ^{2(1-r)} \sum_{\substack{\beta\in \Z_{\ge 1}^r \\ |\beta| = j}} \frac{V_\beta^\pm }{\prod_{k=1}^{r-1} |\beta|_k (j-|\beta|_k)}, \label{eq:421}
\end{align}
where for multiindicies $\beta = (\beta _1 ,\dots ,\beta_r)$, $V_\beta^\pm = V_{\beta_1}^\pm \cdots V_{\beta_r}^\pm$ and $|\beta|_k \coloneq \beta_1 + \cdots + \beta_k$.
Let $\mathcal{RC}^\pm = \bigcup_{j=2}^\infty\mathcal{RC}_j^\pm$ and $\mathcal{RC} = \cup_\pm \mathcal{RC}^\pm$.

For $V(x)$ satisfying \eqref{eq:169}, define
\begin{align}
    \Lambda_\pm   = \set{j \in \Z_{\ge 2} : V\in\mathcal{RC}^\pm _j} .\label{eq:428}
\end{align}
\end{defi}

The resonant condition definition arises from a very special cancellation in constructing outgoing solutions.
Concretely, if $u_\pm(x,\lambda) = e^{\pm i\lambda x } v_\pm (e^{\mp A_\pm x },\lambda)$ is outgoing at $\pm \infty$ with $v_\pm (\bullet , \lambda)$ holomorphic near $0$ (see \Cref{def:outgoing}), then if we choose the normalization of $v_\pm$ as $v_\pm (0 ,\lambda) = \Gamma(1-2i\lambda A_\pm ^{-1})^{-1}$ (which is a natural normalization for ensuring $v_+$ is an entire function of $\lambda$ in the construction of $u_\pm$), then $V\in \mathcal{RC}_j^\pm$ if and only if $u_\pm  (\bullet , -i jA_\pm /2  )$ is identically zero.

\begin{ex}
If $V \in C^\infty (\R; \R)$ (satisfying \eqref{eq:169}) is such that $V(x)\big|_{x\ll-1}= e^{x}$ and $V(x)\big|_{x\gg 1} = e^{-x} + e^{-2x}$ then one can readily check that $A_- = A_+  = 1$, $\Lambda_-  =\emptyset$, and $\Lambda_+ = \set{2}$.
\end{ex}

\begin{ex}[Regge--Wheeler potential in $\mathcal{RC}$]\label{ex:RW in RC}
If $\Lambda = 4 - \sqrt 5$ and $m = \frac{\sqrt 5 - 1 }{6}$, then for $\ell = 1$, $V_\ell \in \mathcal{RC}^+_2$ (recalling \eqref{eq:165}).
\end{ex}

\begin{ex}[P\"oschl--Teller potential in $\mathcal{RC}$]\label{prop:rc3}
The P\"oschl--Teller potential \eqref{eq:PTpotential} is not in $\mathcal{RC}$.
However for the  P\"oschl--Teller potential well, we have
\begin{align}
\frac{-m(m+1)}{\cosh^2 (x)} \in \mathcal{RC}_j^{\pm} \iff m\in \Z_{\ge 1} \text{ and }j\ge m+1.
\end{align}
\end{ex}

Proofs of \Cref{ex:RW in RC} and \Cref{prop:rc3} are in Appendix \ref{section:res_cond_examples}.

\begin{theorem}[Main result for resonant condition potentials]\label{thm:main_rc}
Suppose $V \in C^\infty (\R ; \R)$ satisfies \ref{hyp:1}, $\Lambda_\pm$ are as defined in \eqref{eq:428}, and $U_V (t)$ is the wave propagator for $\Box_V$. 
Then there exist smooth functions $\mathcal{A}_\pm (t)$, $t > 0$ with
\begin{align}
    \mathcal{A}_\pm(t) =- \tfrac 1 2 \sum_{\substack{n \in \Z_{\ge 1}\setminus \Lambda_{\pm}}}e^{-tA_{\pm}n/2 }\label{eq:450}
\end{align}
such that
\begin{align}
\Tr(U_V(t) - U_0(t))  = \tfrac{1}{2}\sum_{\lambda_j \in {\rm{Res}}(V)} e^{-i\lambda_j t} + \mathcal{A}_+ ( t ) + \mathcal{A}_ -( t )  - \tfrac 1 2. \label{eq:507}
\end{align}
\end{theorem}

Note that \Cref{thm:main_rc} is a generalization of the main result. 
It therefore suffices to prove \Cref{thm:main_rc} directly.
For an explicit application of \Cref{thm:main_rc}, see \Cref{prop:581}.

\begin{rem}
Note that if either $\Lambda_\pm = \emptyset$, then (for the corresponding sign) $\mathcal{A}_\pm (t)$ is exactly the same as \eqref{eq:209}.
Moreover, since $\Lambda_\pm \subset \Z_{\ge 2}$, we are guaranteed (by geometric summing in \eqref{eq:450}) that $|\mathcal{A}_\pm(t)|\leq (2(e^{tA_\pm/2}-1))^{-1}$.
\end{rem}

{  A different perspective, relevant to the formulation of our main theorem, can be given using a simple version of {\em Vasy's method} \cite{vasy2013microlocal} -- see \cite{zworski2016resonances} for a detailed but brief presentation and \cite[Chapter V]{dyatlov2019mathematical} for an in-depth textbook introduction and references. }

{  To explain this, we first consider the operator $ - \partial_x^2 + V ( x ) $ on $ [0, \infty ) $, $ V = F ( e^{-Ax } ) $, 
where $ F $ is a {\em smooth} function on $ ( -\varepsilon , 1 ] $, $ \varepsilon > 0 $. 
If we conjugate this operator, the outgoing condition (\Cref{def:outgoing}) 
will eventually become a regularity condition:
\begin{equation}
    \label{eq:conju}
e^{-i \lambda x} (-\partial_x^2 - \lambda^2) e^{i \lambda x}
= -(\partial_x + i\lambda)^2 - \lambda^2 \\
= -\partial_x^2 - 2i\lambda \partial_x.
\end{equation}
We then make a change of variables, $
s = e^{-A x} $ (in some sense it undoes the Regge--Wheeler change of variables but we present it here for general 1D potentials), so that
\[ e^{-i \lambda x} (-\partial_x^2 - \lambda^2) e^{i \lambda x} 
= - (A s \partial_s)^2 + 2i\lambda A s \partial_s \\
= A^2 s \left( -\partial_s s \partial_s + 2i{\lambda} A^{-1}  \partial_s \right) .
\]
We then define the operator
\[ Q_V ( \lambda ) \coloneq A^{-2}s^{-1} e^{-i\lambda x }(-\p_x^2 + V(x))e^{i\lambda x} = -\partial_s s \partial_s + 2i{\lambda} A^{-1}  \partial_s + A^{-2} s^{-1} F ( s ) . \]
If we then consider this operator on $ (-\epsilon , 1 ] $,
it becomes the {\em Vasy operator} in this simplest setting.}

{  In the case of $ P_V $ in Proposition \ref{p:0}, we apply the same construction near $ \pm \infty $ (with the corresponding 
$ A_\pm $ in place of $ A $), and $ s \in ( -\varepsilon, 
1 + \varepsilon ) $  (with $ + \infty $ corresponding to $ 0 $ and
$ - \infty $ to $ 1 $ and $ 1 -s = e^{A_- x }$ for $ 1 -\varepsilon < s < 1$), and an arbitrary change of variables in the middle. This gives $ Q_V ( \lambda )$ which is an elliptic operator 
on $ ( \varepsilon, 1 - \varepsilon ) $, and
\begin{equation}
\label{eq:defQV} Q_V ( \lambda ) = \left\{ \begin{array}{ll}
 -\partial_s s \partial_s + 2i{\lambda} A_+^{-1}  \partial_s + A_+^{-2} s^{-1} F_+  ( s ) ,  &  s \in ( - \varepsilon, \varepsilon ), \\
  -\partial_r r \partial_r + 2i{\lambda} A_-^{-1}  \partial_r + A_-^{-2}
 r^{-1} F_- ( r ) ,  & r := 1-s, \ s \in ( 1 - \varepsilon, 1 + \varepsilon ).
\end{array} \right. \end{equation}
(The parameter $ \varepsilon $ depends on the size of the neighborhoods in which \eqref{hyp:1} is valid.)}

{  We now apply \cite[Theorem 3]{zworski2016resonances}. To state it we recall that
$ \bar H^{s}(a,b) $ denotes restriction to $ (a,b) $ of elements 
of $ H^s ( ( a - \delta, b + \delta )) $ for some $ \delta > 0 $ (that is,
extendable distributions). We then put
\begin{equation}
\label{eq:defXY} \mathscr Y_s:= \bar H^s ( ( -\varepsilon , 1 + \varepsilon ) ) , \ \ \ \mathscr X_s := \{ u \in \mathscr Y_{s+1} : 
Q_V ( 0 )u \in \mathscr Y_s \} . 
\end{equation}
With this notation, we have
\begin{prop}
\label{p:mermaid}
For $ s \in \mathbb R $ and $ {\rm{Im}} \, \lambda > -s - \frac12 $, 
\begin{equation}
\label{eq:Plam}
    \lambda \mapsto Q_V( \lambda )^{-1} \colon \mathscr Y_s \to \mathscr X_s 
    \end{equation}
    is a {\em meromorphic family of operators} with poles of 
    finite rank. In particular, $ \lambda \mapsto Q_V ( \lambda )^{-1}
    \colon \bar C^\infty ( -\varepsilon, 1 + \varepsilon ) \to 
    \bar C^\infty ( -\varepsilon, 1 + \varepsilon ) $ 
    is a meromorphic family of operators.
    \end{prop}
Proposition \ref{p:mermaid} implies Proposition \ref{p:0} (for a more general class of potentials -- see \Cref{rem:611}) as in 
\cite[\S 7]{zworski2016resonances}: there exist operators
\[ T ( \lambda )\colon C_{\rm{c}}^\infty ( \mathbb R_x ) \to C_{\rm{c}}^\infty ( -\varepsilon, 1 + \varepsilon ), \ \ \ U ( \lambda ) \colon \bar C^\infty ( 
( -\varepsilon, 1 + \varepsilon )) \to C^\infty ( \mathbb R ) \]
such that 
\begin{align}
     R_V ( \lambda ) = U ( \lambda ) Q_V ( \lambda )^{-1} T ( \lambda ) . \label{eq:572}
\end{align}
We denote by $ \Res(V) $ the set of poles of
$ R_V (\lambda ) $ with a natural definition of multiplicity -- see Definition~\ref{def:multi}. The set of poles of $ Q_V (\lambda)^{-1} $
is denoted by $ \Res_{A_+,A_-} ( V ) $ with multiplicities defined by \cite[(C.4.6)]{dyatlov2019mathematical}.
By \eqref{eq:572}, we have
\begin{align}
    \Res_V(\lambda ) \subset  \Res_{A_+,A_-} ( V ),  \label{eq:579}
\end{align}
but the inclusion could be strict. The invisibility of a pole of $ Q_V ( \lambda )^{-1}  $ in compact subsets of $ \mathbb R $ (that is, for
$ s \in [ \delta, 1- \delta ] $, $ \delta> 0 $) appears when the co-resonant state of $ Q_V ( \lambda ) $ is supported at infinity, that is at $ s = 0 $ or $ s =1 $. 
}
{In terms of solutions of \eqref{eq:defQV} we have the following characterization:
\[ \lambda \in \Res_{A_+,A_-}  ( V ) \ \Longleftrightarrow \ 
\exists \, w \in C^\infty ( ( -\varepsilon , 1 + \varepsilon ) ), \ \ 
Q_V (\lambda ) w = 0 . 
\]
In terms of $ P_V $ it means that there exists a solution 
\begin{equation} 
\label{eq:resPV} ( P_V - \lambda^2 ) u = 0 , \ \ 
u (x ) = \left\{\begin{array}{ll} e^{ i \lambda x } w_+ ( e^{ -A_+ x } ), & x > R , \\
e^{ - i \lambda x } w_- ( e^{ A_- x } )  , & x < - R , 
\end{array} \right. 
\end{equation}
with $ w_\pm \in C^\infty ( ( - \epsilon, \epsilon ) ) $ and $ R \gg 1 $. As we will see in the proof of \Cref{prop:existenceofoutgoing}
\[ \Res_{A_+,A_-}  ( V ) \setminus \Res ( V )  \subset - \tfrac12 i (  A_+ \Z_{\ge 2}
\cup  A_- \Z_{\ge 2}) . \]
Resonances in $\Res_{A_+,A_-}  ( V ) \setminus \Res ( V )$ exactly correspond to elements of $\Lambda_\pm$ \eqref{eq:428}.

}

{ To see the mechanism for the existence of such resonances, we go back to the simple model of a half-line with $ F \equiv 0 $ and the Dirichlet boundary condition at $ x = 0 $ (that is $ s = 1 $). The adjoint, 
\[ Q_0(\lambda)^* = -\partial_s (s \partial_s) + 2i A^{-1} \bar {\lambda}{\partial_s},
\]
applied to $ \delta ( s ) $ gives $ 
( 1 + 2 i A^{-1} \bar \lambda ) \delta' ( s )$. This vanishes when
$ \lambda = - \frac12 i A $. 
The resonant state, that is the {\em smooth} solution to $ 
Q_0 ( \lambda ) u = 0  $ with $ u ( 1 ) = 0$, is given by 
$ u (s) = 1 - s   $. For more on co-resonant states supported at infinity and the resulting confusion in the physics literature see
\cite{hintz2022quasinormal} and references given there.}

{  
\begin{rem}\label{rem:611}
As noted above, meromorphic continuation of $ Q_V (\lambda )^{-1} $, and hence of $ R_V ( \lambda)  $, holds under weaker assumptions than
\eqref{hyp:1}: only smoothness is needed for $ F_\pm $. For the main goal of the paper, that is, for establishing a Poisson formula for resonances, we need global estimates for the scattering matrix. That seems to require stronger assumptions. A similar issue has been encountered by J\'ez\'equel \cite{Jezequel_globatrace}
in the study of global trace formula for Ruelle--Pollicott resonances. 
Analyticity was needed to replace the ``strip-by-strip" continuation seen
in Proposition \ref{p:mermaid} by a construction giving global estimates.
\end{rem}}

Under the same hypotheses as \Cref{thm:main_rc}, we can therefore write
\begin{equation}\label{eq:626}
\begin{split}
       &\Tr(U_V (t) - U_0(t)) \\
       &\quad\quad= \tfrac{1}{2} \sum_{\lambda_j\in \Res_{A_+,A_-}(V)} e^{-i\lambda _j t }- \tfrac 1 2 \left( \frac{1}{e^{tA_+ /2 } -1} + \frac{1}{e^{tA_- /2} - 1} \right )-\tfrac 1 2 ,
\end{split}
\end{equation}
noting that if $V \notin \mathcal{RC}$, then $\Res_{A_+,A_-}(V) = \Res (V)$ and \eqref{eq:626} is just a restatement of our main result.

We remark that in some sense, it is generic that a potential not lie in $\mathcal{RC}$. More precisely, we have the following qualitative statement.

\begin{prop}[Generically $V \notin \mathcal{RC}$]\label{prop:rc1}
The set of potentials satisfying \eqref{hyp:1} defined by the coefficients $V_j^\pm$ in \eqref{eq:169} generically satisfies $V\notin\mathcal{R}\mathcal{C}$ in the sense that the set of sequences $(V_j^\pm)_{j\geq 1}$ of coefficients satisfying the resonant condition \eqref{eq:421} for some $j\in\mathbb{Z}_{\geq 2}$ is a meagre (i.e., a countable union of nowhere dense subsets) subset of the space of complex sequence $\mathbb{C}^{\mathbb{N}}$ with the product topology.
\end{prop}
\begin{proof}
The proof is simple. We have $V\in\mathcal{R}\mathcal{C}$ if and only if $V\in\bigcup_{j\geq 2}\mathcal{R}\mathcal{C}_j^\pm$. But $V$ belongs to $\mathcal{R}\mathcal{C}_j^\pm$ for some $j$ if and only if the sequence of coefficients $(V_j^\pm)_{j\geq 1}\subset Z_j\times \mathbb{C}^{\mathbb{Z}_{\geq j+1}}\subset\mathbb{C}^{\mathbb{N}}$ where $Z_j$ is the set of coefficients $(V^\pm_1,...,V^\pm_j)\in \mathbb{C}^j$ satisfying \eqref{eq:421}. Since $Z_j$ is the zero set of a polynomial in $\mathbb{C}^j$, it is nowhere dense in $\mathbb{C}^j$. Therefore, $Z_j\times \mathbb{C}^{\mathbb{Z}_{\geq j+1}}$ is nowhere dense in $\mathbb{C}^{\mathbb{N}}$.
\end{proof}

We expect that, outside of a measure zero set (as in \Cref{prop:rc3}), the potentials $V_\ell $ from the Regge--Wheeler change of coordinates are not in $\mathcal {RC}$, but we present this as a conjecture.

\begin{conj}
For each $\Lambda>0$, the measure of the set
\begin{align}
    \set{m \in \left(0,\frac{1}{3\sqrt \Lambda}\right) :   \exists\  \ell \in \Z_{\ge 0} \ \text{s.t.} \ V_\ell (x;m,\Lambda) \in \mathcal{RC} }
\end{align}
is zero where $V_\ell$ are defined in \eqref{eq:165}.
\end{conj}

\begin{ex} \label{prop:581}
Here we present an explicit example of our main result for a potential in $\mathcal{RC}$.

For $m >  0$, define
\begin{align}
    V_m \coloneq \frac{-m(m+1)}{\cosh^2(x)}.
\end{align}
By \cite{cevik2016resonances} and \Cref{ex:PT:resonances}
\begin{align}\label{eq:695}
    \Res (V_m )  =\begin{cases}
        \set{i(m-n) : n \in \Z_{\ge 0}}  \cup \set{-i(m+1+n) : n \in \Z_{\ge 0}}, &  \text{ if }m \in \R_{> 0} \setminus \Z,\\
        \set{i n : n\in \Z \cap [-m,m] }, & \text{ if }m\in \Z_{\ge 1}. 
    \end{cases}
\end{align}
(where in the non-integer case, elements in the intersection of the two sets are resonances with multiplicity two.)

If $m  \in \R_{> 0} \setminus \Z$, we can compute
\begin{align}
    \sum_{\lambda _j \in \Res{V_m}} e^{-i\lambda _j t }= \frac{e^{mt } + e^{-(m+1)t} }{1-e^{-t}} = \frac{\cosh((m+1/2)t)}{\sinh(t/2)}.
\end{align}
While, because $V_m\notin \mathcal{RC}$ (\Cref{prop:rc3}),
\begin{align}
    \mathcal{A}_+(t) + \mathcal A_-(t) = - \frac{e^{-t/2}}{2\sinh(t/2)}
\end{align}
so that our main result \eqref{eq:218} gives us
\begin{align}
    \tr(U_{V_m} (t) - U_0(t)) =\frac{\cosh((m+1/2)t) - e^{-t/2}}{2\sinh(t/2)} -\frac{1}{2}. \label{eq:655}
\end{align}
On the other hand, if 
$m\in \Z_{\ge 1}$, then (by \eqref{eq:695}) 
\begin{align}
    \sum_{\lambda_j \in \Res{V_m}} e^{-i\lambda _j t} = \sum_{j=-m}^{m} e^{jt}
\end{align}
and by \Cref{prop:rc3}, $\Lambda_\pm = \Z\cap [m+1,\infty)$ so that 
\begin{align}
    \mathcal{A_+}(t) = \mathcal{A}_{-}(t) = -\tfrac 1 2 \sum_{j=1}^{m} e^{-tj}
\end{align}
therefore by \Cref{thm:main_rc}
\begin{align}
    \tr(U_{V_m}(t) - U_0(t) ) = \tfrac{1}{2}\sum_{j=1}^{m} e^{jt} - \tfrac  1 2 \sum_{j=1}^{m} e^{-jt} = \sum_{j=1}^{m} \sinh (jt). \label{eq:668}
\end{align}
Combining \eqref{eq:655} and \eqref{eq:668}, we get (for $t > 0$,
\begin{align}
    \tr(U_{V_m} (t)  - U_0 (t)) = \begin{cases} \frac{\cosh((m+1/2)t) - e^{-t/2}}{2\sinh(t/2)} -\frac{1}{2}, & \text{if } m \in \R_{>0} \setminus \Z, \\
    \sum_{j=1}^{m} \sinh (jt), &\text{if } m\in \Z_{\ge 1}.
    \end{cases}
\end{align}
One can verify that this is a continuous function in $m$ suggesting that, while resonances can appear and disappear with small perturbations of the potential, the trace formula is stable, which is captured by our resonant condition definition.
\end{ex}

}

\subsection{Outline of proof}
An outline of a proof of our main result is presented below.
The order of steps below is \textit{not} in the chronological order of the paper, but presented in a way that motivates the required analysis of \S \ref{s:analysisofincoming}.

\begin{enumerate}
    \item The Birman--Kre\u{\i}n trace formula (which we prove for our class of potentials in \Cref{appendix}) relates the trace of $U_{V}(t)-U_0(t)$ to a term depending on negative eigenvalues of $P_V$, a term involving the resonance of $V$ at $0$ (if it exists), and an integral involving the logarithmic derivative of the scattering matrix $S(\lambda)$ for $V$ (see \eqref{BK1}). 
    This is presented in \S \ref{s:birman}.
    \item To apply the Birman--Kre\u{\i}n trace formula, we first need to prove the resolvent $(P_V -\lambda^2)^{-1}$ (defined for $\Im \lambda \gg 1$) can be meromorphically continued to $\C$. 
    This is done by defining and estimating incoming and outgoing solutions $u_\pm (\lambda,x)$ to $(P_V -\lambda^2 )u = 0$ (\S \ref{s:incoutgo})
    \item The logarithmic derivative of the scattering matrix is $ \frac{d}{d\lambda} \log(T(\lambda)/T(-\lambda))$ where $T(\lambda)$ is the transmission coefficient.
    \item The transmission coefficient is computed by taking the Wronskian of the incoming and outgoing solutions (see \eqref{eq:598}).
    \item The term involving the logarithmic derivative of the scattering matrix will have singularities coming from $T(\lambda)$, which are decoupled into three sets (see \eqref{eq:887}). 
    Two of these end up giving the terms $\mathcal{A}_\pm(t)$, while the third term is related to the resonances of $V$ (a Lemma relating zeros of the Wronskian of $u_\pm$ to resonances of $V$, with multiplicity, is required -- which is proven in Appendix \ref{s:proof of lemmainduct}).
\end{enumerate}
 
\smallsection{Notation}
We use the following notation throughout this paper. We write $D_x \coloneq -i \partial_x  $. 
We let $L^2_c$ denote the space of $L^2$ functions with compact support and $L^2_{\rm{loc}}$ the space of locally $L^2$ functions
We let $\norm{\cdot}_1$ denote the trace norm.
For functions $f(x)$ and $g(x)$, we write $f(x) \lesssim g(x)$ if there exists a constant $C>0$ such that $f(x) \le  C g(x)$ for all $x$.
If $C$ depends on a parameter $\alpha$, we write $f(x) \lesssim_\alpha g(x)$.
We write $f(x) \gg 1$ to mean $f(x) >  C $ for some sufficiently large $C >0$.
If $g(x)$ is positive and $N \in \Z$, we write $f = \mathcal{O}(g^N)$ to mean there exists $C_N > 0$ such that $|f(x)|\le C_N g(x)^N$.
We write $f = \mathcal{O}(g^{ \infty})$ if $f = \mathcal{O}(g^N)$ for all $N> 0$ (and similarly $f = \mathcal{O}(g^{- \infty}) \iff f = \mathcal{O}(g^{-N}) \ \forall \ N \in \Z_{>0}$).

\smallsection{Acknowledgments}
We thank Maciej Zworski for suggesting this project and numerous helpful discussions (in particular, his help in generalizing the Birman--Kre\u{\i}n trace formula).
We also thank Peter Hintz for helpful comments about the trace formula in the Schwarzschild-de Sitter setting.
Conversations with Kiril Datchev, Malo J\'ez\'equel, Ethan Sussman, and Jared Wunsch were also very useful.
We are also grateful for an anonymous referee report that led to the development of the definition of resonant conditions.
The first author was supported by NSF grant DMS-2136217. The second author was supported by a fellowship
in the Simons Society of Fellows.

\section{Analysis of incoming and outgoing solutions}\label{s:analysisofincoming}

In this section, we first construct and analyze the behavior of the incoming and outgoing solutions to 
\begin{align}
    (D_x^2 + V (x)) u(x) = \lambda ^2 u(x)  
\end{align}
where $V$ has the asymptotics described by \Cref{hyp:1} and $D_x \coloneq - i \p_x$. We will then use this to show that the corresponding resolvent operator $(P_V-\lambda^2)^{-1}$ can be meromorphically extended to the complex plane.

\subsection{Incoming and outgoing solutions} \label{s:incoutgo}
We begin by recalling the definition for incoming and outgoing solutions for $P_V\coloneq (D_x^2+V(x)-\lambda^2)$. In the sequel, we assume $V$ satisfies \Cref{hyp:1} with $A_\pm $.

\begin{defi}[Outgoing solution]\label{def:outgoing}
    We say $u(x)= u(x,\lambda)$ is an outgoing solution for the operator
    \begin{align}
        P_\lambda \coloneqq (D_x^2 + V (x) - \lambda^2) \label{eq:plambdadef}
    \end{align}
    at $+\infty $ if $P_\lambda u  = 0$ and there exists a globally defined, smooth profile $v_+ (w) = v_+(w,\lambda)$ that is holomorphic in a neighborhood of zero, such that
    \begin{align}
        u (x) = e^{i \lambda x }v_+ (e^{-A_+ x}).
    \end{align}
    Analogously, $u(x) = u(x,\lambda)$ is an outgoing solution for $P_\lambda$ at $-\infty $ if $P_\lambda u  = 0$ and there exists a globally defined, smooth profile $v_- (w)$ that is holomorphic in a neighborhood of zero, such that
    \begin{align}
        u (x) = e^{-i \lambda x }v_- (e^{A_- x}).
    \end{align}
\end{defi}

We also have a similar definition for incoming solutions.
\begin{defi}[Incoming solution]\label{def:incoming}
    We say $u(x) = u(x,\lambda)$ is an incoming solution to $P_\lambda$ at $+\infty $ if $P_\lambda u  = 0$ and there exists a globally defined, smooth profile $w_+ (w)$ that is holomorphic in a neighborhood of zero, such that
    \begin{align}
        u (x) = e^{-i \lambda x }w_+ (e^{-A_+ x}).
    \end{align}
    Analogously, $u(x) = u(x,\lambda)$ is an incoming solution to $P_\lambda$ at $-\infty $ if $P_\lambda u  = 0$ and there exists a globally defined, smooth profile $w_- (w)$ that is holomorphic in a neighborhood of zero, such that
    \begin{align}
        u (x) = e^{i \lambda x }w_- (e^{A_- x}).
    \end{align}
\end{defi}
Our first result in this section shows that for each $\lambda$, we can construct outgoing solutions with the asymptotics described above.
\begin{prop}\label{prop:existenceofoutgoing}
For each $\lambda\in \C$, there exist outgoing solutions for $P_\lambda$ of the form
    \begin{align}\label{eq:336}
        u_\pm (x) =e^{\pm i \lambda x } v_\pm (e^{\mp A_\pm x})
    \end{align}
    where $v_\pm(w) = v_\pm (w,\lambda)$ are holomorphic near $w = 0$ in a neighborhood that is independent of $\lambda$,  {not identically zero,} and  {such that
\[         v_\pm (0) = \frac{1}{\Gamma_{\Lambda_\pm}(1-2i\lambda A_\pm^{-1})}  \]
where 
\begin{align}
    \Gamma_{\Lambda_\pm}(z)\coloneq \Gamma(z)\prod_{j\in \Lambda_\pm}\left(1+\frac{z}{j-1}\right)e^{-\frac{z}{j-1}}.\label{eq:547}
\end{align}
That is, we take the Hadamard factorization of the reciprocal $\Gamma$ function, and we remove all of the canonical factors in the set $\Lambda_\pm$.\footnote{Note that if $\Lambda_\pm = \emptyset$, then $\Gamma_{\Lambda_\pm} = \Gamma$. Moreover, note that the poles of $\Gamma_{\Lambda_\pm}(z)$ are first order and are at the points $\set{j\in \Z_{\le 0} : 1 - j \notin \Lambda_\pm}$.}
}
Moreover, there exists $C,M_\pm > 0$, independent of $\lambda$, such that for $x\in \R$
\begin{equation}
        \abs{\partial_x^j v_\pm (e^{\mp A_\pm x})} \le C \exp\Bigl(C \langle x-M_\pm\rangle\langle\lambda\rangle\log \langle\lambda\rangle \Bigr) , \ \ \  j=0,1. \label{eq:boundonvp}    
    \end{equation}
where we write $\langle z \rangle\coloneq (1+|z|^2)^{\frac{1}{2}}$.
\end{prop}

\begin{proof} We will provide the details for the case $u_+$ as the other case follows from analogous reasoning. 

\Step We begin by making an ansatz
\begin{equation*}
u_+(x) = e^{i\lambda x}v_+(e^{-A_+ x}).
\end{equation*}
It will be convenient in the forthcoming calculations to write as shorthand $w\coloneq e^{-A_+x}$. A straightforward calculation shows that solving $P_{\lambda}u_+=0$ is equivalent to $v_+$ solving the equation
 \begin{align}
    2 \lambda D_x v_+ (w)+ D_x^2 v_+(w)+ V (x)v_+(w)=0. \label{eq:121}
 \end{align}
We can write this equation in terms of of $D_w$ by noting the simple identities $D_x = (-A_+ w ) D_w$ and $D_x^2 =   A_+^2 (w D_w )^2 $. Hence, we find that $  (D_x^2 + V(x) -\lambda^2 ) u_+(x) =0$ if and only if
\begin{align}
    -2\lambda A_+w  D_w v_+ +A_+^2 (wD_w)^2 v_+ + V v_+ = 0. \label{eq:251}
\end{align}
To construct $v_+$, we will begin by formally expanding $v_+$ and $V$ in a power series near $w=0$:
\begin{align}
     v_+(w) = \sum_{j=0}^\infty v_j w^j , \ \ 
    \ V(w) = \sum_{j=1}^\infty V_j w^j,  \label{eq:510}
\end{align}
where the coefficients $V_j$ are defined as in \eqref{hyp:1}. We note that the potential does not have a constant term in this expansion (i.e., $V$ vanishes to first order at $w=0$).

Substituting this into \eqref{eq:251} and matching powers of $w$, we obtain the follow recursion relation for each $j \ge 1$:
\begin{align}
    j A_+^2 (j - 2 i A_+^{-1} \lambda) v_j  = \sum_{\ell = 1}^j V_{\ell}v_{j - \ell}. \label{eq:258}
\end{align}

\Step We now prove the $v_j$'s (the coefficients in the Taylor expansion of $v_+$) have sufficient decay to ensure that $v_+$ is analytic and has a nontrivial radius of convergence that is uniform in $\lambda$. 
We will begin by showing that the coefficients $v_j$ defined by the recursion \eqref{eq:258} are entire functions of $\lambda$,  {and are not all zero}.
 
 {
To simplify notation, in the forthcoming calculations, we set
\begin{align}
    \alpha \coloneq 2iA_+^{-1}\lambda \label{eq:defofalpha}
\end{align}
and define
\begin{align}
    T_j = T_j (\alpha ) \coloneq j A_+^2 (j-\alpha), \quad S_j = S_j (\alpha ) \coloneq \sum_{\ell = 1}^ jV_\ell v_{j - \ell}.
\end{align}
Then \eqref{eq:258} can be rewritten as $T_j v_j = S_j$ so that formally $v_j = T_j ^{-1} S_j$.
We would like to formally write $S_j$ as only involving $v_0$.
For each $j$, we can write
\begin{align}
    S_j  &= \sum_{\ell_1 = 1}^j V_{\ell_1} v_{j-\ell_1} =V_j v_0 +\sum_{\ell_1 = 1}^{j-1} V_{\ell_1} T_{j-\ell_1}^{-1} S_{j-\ell_1} \\
    &= V_j v_0 + \sum_{\ell_1 + \ell _2 = j}V_{\ell_1} V_{\ell_2} T_{j-\ell_1}^{-1} v_0 + \sum_{\ell_1 = 1}^{j-1} V_{\ell_1} T_{j-\ell_1}^{-1}\sum_{\ell_2 = 1}^{j-\ell_1-1} V_{\ell_2} v_{j-\ell_1-\ell_2}. 
\end{align}
We can repeat this until $S_j$ is written as a sum involving only $v_0$'s:
\begin{align}
    S_j  = V_j v_0 + \sum_{r=2}^j \sum_{\substack{\beta\in \Z_{\ge 1}^r\\ |\beta |= j}}  \frac{V_\beta}{\prod_{k=1}^{r-1} T_{j-|\beta|_k}} v_0\label{eq:612}
\end{align}
where $V_\beta = V_{\beta_1}V_{\beta_2}\cdots V_{\beta_r}$, and $|\beta|_k \coloneq \beta_1 + \beta_2 + \cdots +\beta_k$. 
A straightforward computation shows that 
\begin{align}
\sum_{r=2}^j \sum_{\substack{\beta\in \Z_{\ge 1}^r\\ |\beta |= j}}  \frac{V_\beta}{\prod_{k=1}^{r-1} T_{j-|\beta|_k}(j)}=  - \sum_{r=2}^j (-1)^{r}A_+^{2(1-r)} \sum_{\substack{\beta \in \Z_{\ge 1}^r \\ |\beta | = j}}\frac{V_\beta}{ \prod_{k=1}^{r-1} |\beta|_k  (j-|\beta|_k)} \label{eq:875}
\end{align}
which is exactly the expression appearing in the resonant condition definition \eqref{eq:421}.

\begin{rem}
Here is \textit{exactly} where the resonant condition shows up.
The formal relation $v_j = T_{j}^{-1} S_j$ can be satisfied by choosing the correct normalization of $v_0$ to cancel any poles of $T_j^{-1}$.
The term $T_j^{-1}$ has a single pole of order one at $\alpha = j$.
Therefore, setting $v_0 = \Gamma(1-\alpha)^{-1}$ will cancel any pole.
The subtlety is, if $S_{j_0} (j_0) v_0^{-1}$ vanishes, then all the $v_j$'s are zero at $\alpha = j_0$.\footnote{An argument can then be used to show that $u_-$ is outgoing at both infinities at frequency $\alpha = j_0$ -- suggesting that $\lambda = -ij_0A_+^{-1}/2$ is a resonance. However, $\lambda$ is \textit{not} a pole of the resolvent (it is a pole of the inverse of the Vasy operator $Q(\lambda)$ (recall \Cref{p:mermaid}). Note that our definition of outgoing solutions involving $A_\pm$ was somewhat artificial, unlike the case for compactly supported potentials.}
This vanishing of $S_j(j)v_0^{-1}$ is the resonant condition definition \eqref{eq:421}.
\end{rem} 

We now claim that by setting $v_0 = \Gamma_{\Lambda_+}(1-\alpha)^{-1}$ (recall \eqref{eq:547}), the $v_j$'s are entire functions of $\alpha$ and either $v_0(j)\neq 0$ or $v_j (j) \neq 0$.
This follows by induction on $j$. 
For $j=1$, we have $v_1 (\alpha)= V_1 T_1(\alpha)^{-1} v_0(\alpha)$.
This is an entire function of $\alpha$ as the only pole of $T_j$ at $\alpha = 1$ is canceled by the vanishing (of the same order) of $v_0(\alpha)$ and $V_1\neq 0$ ensures that $v_1(1) \neq 0$.
Assume now that all $v_j$'s are entire functions of $\alpha$, then $S_{j+1}$ is an entire function of $\alpha$.
Therefore we must only check that $v_{j+1}$ is holomorphic at $\alpha = j+1$ (from the formula $v_{j+1} = T_{j+1}^{-1} S_{j+1}$).
But now we can write
\begin{align}
    v_{j+1} (j+1) = \lim_{\alpha \to j+1}  \underbrace{(T_{j+1}(\alpha)^{-1})}_{\text{I}}\underbrace{(V_{j+1}+\sum_{r=2}^{j+1}\sum_{\substack{\beta\in \Z_{\ge 1}^r\\ |\beta |= j+1}}  \frac{V_\beta}{\prod_{k=1}^{r-1} T_{j+1-|\beta|_k}(\alpha)} )}_{\text{II}}\underbrace{v_0(\alpha)}_{\text{III}}.
\end{align}
If $j+1 \in \Lambda_+$, then $\text{III}\neq 0$, $\text{II}$ vanishes at $\alpha = j+1$ by the resonant condition definition (using \eqref{eq:875}), and $T_{j+1}$ has a pole of order one at $j+1$ canceling the zero, so that $v_{j+1}$ has no pole at $j+1$.

If $j+1 \notin \Lambda_+$, then $\text{II}$ is nonzero at $\alpha = j+1$, and $\text{III}$ vanishes to first order at $j+1$, canceling the first order pole of $\text{I}$ of first order at $\alpha = j+1$, thus proving $v_{j+1}$ has no pole at $j+1$ and is nonzero there.

}

 {
We next show that for all $\alpha\in \C$, $v_j$'s are not all zero.
First, if $\alpha \in  (\C \setminus \Z_{\ge 1})\cup \Lambda_+$, then $v_0 (\alpha) = \Gamma_{\Lambda_+}(1-\alpha)^{-1} \neq 0$.
Next, if $\alpha = 1$, we see from \eqref{eq:258} that $v_1 (1)=  \lim_{\alpha \to 1} (1- \alpha )^{-1}A_+^{-2} V_1 v_0 (\alpha)\neq 0$ as we assumed $V_1 \neq 0$.
Finally, if $\alpha \in \Z _{\ge 2} \setminus \Lambda_+ $, then we have shown above that $v_\alpha (\alpha ) \neq 0$.

}

\Step For each $j$, we want to obtain a strong enough growth bound for $v_j(\alpha)$ for $\alpha \in \C$. We will do this in two steps. First, we will estimate $v_j(\alpha)$ for $\alpha$ away from $\N$, and then extend this to a global bound by using the maximum principle.

We now use the assumption that $|V_j|\le A^j$ for some positive constant $A$. The following lemma will give us the required estimate away from $\mathbb{N}$.

\begin{lemma}\label{lemma:1}
  Suppose $\e \in (0,A_+^{-2})$ and $\Omega_\e \coloneq \set{z \in \C : \dist(z,\N) > \e}$, then for $\alpha \in \Omega_\e$,
\begin{align}
    \abs{v_j(\alpha)} \le \frac{A^{j}\abs{v_0}}{\e^{j}A_+^{2j}} .
\end{align}  
\end{lemma}
\begin{proof}
    We prove this by induction. For $j=0$ this is clear. We then have
\begin{align}
    \abs{v_j(\alpha) } &\le  \frac{1}{j A_+^2|j-\alpha|} \abs{\sum_{k=0}^{j-1} v_k V_{j-k} }\le \frac{\abs{v_0}}{j A_+^2 \e} \sum_{k=0}^{j-1} A^{k} \e^{-k} A_+^{-2k}A^{j- k}\\
    &= \frac{ |v_0|  A^{j}}{jA_+^{2}\e} \sum_{k=0}^{j-1} \e^{-k}A_+^{-2k}\le \frac{|v_0|A^j}{j A_+^2 \e} \sum_{k=0}^{j-1}\e^{-(j-1)} A_+^{-2(j-1)} = \frac{|v_0|A^j}{\e^j A_+^{2j}}.
\end{align}
The first inequality follows from \eqref{eq:258}, while the second inequality follows from the induction hypothesis, and the third inequality follows from the choice of $\e$.
\end{proof}

Next, we extend Lemma \ref{lemma:1} to a global estimate. We fix $\e _ 0\in (0,\min(A_+^{-2}, 1))$. 
By Lemma \ref{lemma:1}, and the choice of  {$v_0 = \Gamma_{\Lambda_+}(1-\alpha)^{-1}$} and the above steps, we have that $v_j(\alpha)$ are entire functions, and for $\alpha\in \Omega_{\e_0}$, we have
 {
\begin{align}\label{eq:28}
    \abs{v_j(\alpha)}\le \frac{ A^{j}}{\e_0^{j} A_+^{2j} \abs{\Gamma_{\Lambda_+}(1-\alpha)}}.
\end{align}
}

To obtain a global bound for $v_j(\alpha)$, we use the following standard bound of the reciprocal of the Gamma function. It follows from the factorization of $ \Gamma( z)$ and, for instance, \cite[Theorem 1.11]{Hayman1964}.  {This will enable us to obtain a bound on the reciprocal of the modified $\Gamma_{\Lambda_+}$}

\begin{lemma}\label{lemma:gammabound}
    There exist a constant $C > 0$ such that for $z\in \C $,
    \begin{align}
        \abs{\frac{1}{\Gamma(z)}} \le C\exp\left( C \abs z \log (1+ \abs{z}) \right) .
    \end{align}
\end{lemma}
 { Since $\frac{1}{\Gamma_{\Lambda_+}}$ is obtained from $\frac{1}{\Gamma}$ simply by removing a subset of its canonical zero factors, it is an entire function and it has the same growth bound as $\frac{1}{\Gamma}$ (with constants uniform in the choice of subset). Therefore, we have the corresponding lemma for the modified $\Gamma$ function,
\begin{lemma}\label{lemma:modgammabound}
There exists a constant $C>0$ such that for all $z\in\mathbb{C}$,
\begin{equation*}
 \abs{\frac{1}{\Gamma_{\Lambda_+}(z)}} \le C\exp\left( C \abs z \log (1+ \abs{z}) \right).    
\end{equation*}
\end{lemma}
}
Now, for $k\in \N$, we apply the maximum principle and  {Lemma \ref{lemma:modgammabound}} to get that:
\begin{align}
    \max_{|\alpha - k| \le \e_0 } \abs{v_j (\alpha) }&\le \max_{|\alpha - k| = \e_0} \abs{\frac{A^{j}}{\e_0^{j}A_+^{2j} {\Gamma_{\Lambda_+}(1-\alpha)}}}\le \max_{|\alpha - k| = \e_0} \frac{ CA^{j}}{A_+^{2j}\e_0^{j}} e^{C |1-\alpha| \log (|1-\alpha| +1 )} \\
    &\le \frac{C A^{j}}{A_+^{2j}\e_0^{j}} e^{C (k+1) \log(k+2)} \le  \frac{C A^{j}}{A_+^{2j}\e_0^{j}} e^{C k \log(k)}\label{eq:394}
\end{align}
where the constant $C$ may change in each inequality.

We can then use this to get a global bound on $v_j$. We have by \eqref{eq:28} that for $\dist(\alpha,\N) > \e_0$:
\begin{align}
    \abs{v_j (\alpha) }\le \frac{ A^{j}}{\e_0^jA_+^{2j}} e^{C |1 - \alpha | \log (|1-\alpha | + 1))}
\end{align}
and by \eqref{eq:394}, for $|\alpha  - k |\le \e_0$ (for some $k\in \N$):
\begin{align}
  \abs{  v_j(\alpha)} \le  \frac{CA^j}{A_+^{2j}\e_0^j} e^{C k \log (k)}.
\end{align}
Therefore, there exists a $C>0$ such that for all $\alpha \in \C$:
\begin{align}
    \abs{v_j (\alpha )}\le \frac{CA^j}{A_+^{2j}\e_0^j} e^{C (|\alpha| + 1) \log (|1-\alpha| +1)}.
\end{align}
This shows that $v_+(w)$ as defined by \eqref{eq:510} is well-defined in some neighborhood of zero with radius of convergence that is independent of $\lambda$.
\\
\Step A priori, the above construction only ensures that $v_+$ is defined in a neighborhood of $w=0$ (uniformly in $\lambda$). We now carry out a standard energy-type estimate for $P_{\lambda}$ to show that $v_+$ can be extended to a global (in $w$) smooth solution of \eqref{eq:251}. Along the way, we also show the bound \eqref{eq:boundonvp}. We begin by fixing $\delta > 0$ less than the radius of convergence of $\sum_0^\infty v_j w^j$ (which is independent of $\alpha$), to ensure that there exists a $C  = C(\delta )$ (independent of $\alpha$) such that
\begin{align}
    \abs{(\partial_w^j v_+)(\delta) } \le C e^{ C (|\alpha | + 1 ) \log (|1-\alpha | +1)}\label{eq:476}
\end{align}
for $j = 0,1$. We recall that in terms of the original variable $x$, we have $x(\delta) = A_+^{-1} \log(1/\delta)$.
We will now use a simple energy estimate to estimate the growth of $v_+$ and its derivatives.
By \eqref{eq:121}, $v_+$ satisfies the the ODE in the $x$ variable:
\begin{align}
    2 i \lambda \p_x v_+(w) + \p_x^2 v_+(w) - V(x) v_+(w) = 0.
\end{align}
Here, we can rewrite the above equation as a $2\times 2$ system
\begin{equation}\label{matrixeqn}
\p_x \mat{u_1(x) \\ u_2(x)} = \underbrace{\mat{0 & 1 \\ {V(x)} & -2i\lambda }}_{\coloneqq \textbf{A}(x)}\mat{u_1(x) \\ u_2(x)}.
\end{equation}
where $u_1 (x)\coloneq v_+(w(x))$ and $u_2 (x)\coloneq \p_x (v_+(w(x))$. Let $\textbf{U}(x) \coloneq (u_1 (x),u_2(x))^t$, so that $\p_x \textbf{U} = \textbf{A}\textbf{U}$. 
This gives the simple energy identity
\begin{align}
    \frac{d}{dx} \abs{\textbf{U}}^2 = 2 \Re{\ip{\textbf{A} \textbf{U}}{\textbf{U}}}.
\end{align}
By the Cauchy--Schwartz inequality 
\begin{align}
     \frac{d}{dx} \abs{\textbf{U}(x)}^2 \ge -C\langle\lambda\rangle\abs{\textbf{U}(x)}^2
\end{align}
for some positive constant $C$ (independent of $\lambda$). This gives us that:
\begin{align}
    \frac{d}{dx} \left ( e^{C\langle\lambda\rangle x} \abs{\textbf{U}(x)}^2\right ) \ge 0
\end{align}
so that integrating from $x$ to some $x_0>x$, we get that:
\begin{align}
    \abs{\textbf{U}(x)}^2 \le e^{C\langle\lambda\rangle(x_0 - x)} \abs{\textbf{U}(x_0)}^2.
\end{align}
Letting $x_0 = x(\delta)$, we get, by \eqref{eq:476} that there exists a $C>0$ such that for $j= 0,1$
\begin{align}
    \abs{\p_x^j (v_+ (w(x)) } \le C  e^{C\langle\lambda\rangle (x(\delta)- x) + C(|\alpha| + 1 ) \log(|1-\alpha| + 1)}.
\end{align}
This shows that $v_+$ can be continued to a global $C^1$ (in $w$) solution of $P_{\lambda}v_+=0$ and that (by elementary estimates) $v_+$ satisfies the bound \eqref{eq:boundonvp}. To show that $v_+$ is smooth, one can simply apply derivatives to \eqref{matrixeqn} and inductively apply the above argument.
\setcounter{step}{0}
\end{proof}

\subsection{Meromorphic continuation of the resolvent}\label{s:meromorphic continuation}
We now prove a more precise version of Proposition \ref{p:0} which meromorphically extends the resolvent $R_V (\lambda)$ to the entire complex plane.
The poles of the meromorphically continued resolvent will be the zeros of the Wronskian of outgoing solutions to $P_V -\lambda^2$.
This Wronskian is entire in $\lambda$.
A technical Lemma is first required to prove that the Wronskian is not identically zero.

\begin{prop}\label{prop:Wronskian}
There exist purely imaginary $\lambda$ in the upper half plane such that the outgoing solutions $u_+(x,\lambda)$ and $u_-(x,\lambda)$ from Proposition \ref{prop:existenceofoutgoing} are linearly independent. 
In particular, the Wronskian of $u_+(x,\lambda)$ and $u_-(x,\lambda)$
\begin{align}
   W[u_+,u_-](\lambda) \coloneq  (\p_x u_+(x,\lambda) )u_-(x,\lambda) - u_+(x,\lambda) (\p_x u_-(x,\lambda)) \label{eq:wronk}
\end{align}

is not identically zero in $\lambda$.
\end{prop}
\begin{proof}
If we suppose the Wronskian is identically zero, then $u_+(x,\lambda)$ and
$u_-(x,\lambda)$ are linearly dependent for all $\lambda$.
Setting $\lambda=i\sigma$ for $\sigma\in\mathbb{R}_{>0}$, we have
$u_+(x,i\sigma)=c(\sigma)\,u_-(x,i\sigma)$. Since $u_+$ is outgoing at $+\infty$
and $u_-$ is outgoing at $-\infty$, this implies
$|u_+(x,i\sigma)|\lesssim e^{-\sigma x}$ for $x\gg 1$ and
$|u_+(x,i\sigma)|\lesssim e^{\sigma x}$ for $x\ll -1$.
Hence $u_+(\bullet,i\sigma)\in L^2(\mathbb{R})$, so $-\sigma^2$ is an eigenvalue
of $P_V$ for all $\sigma>0$.
But the negative spectrum of $P_V$ is discrete, so it cannot contain the
continuum $\{-\sigma^2:\sigma>0\}$; this is a contradiction.

\end{proof}
As a consequence of the proposition, we have the following:
\begin{prop}\label{prop:meroresolv}
For $\Im{\lambda} \gg 1$, there exists an operator
\begin{align}
    R_V(\lambda) \coloneq (D_x^2 + V(x) -\lambda^2) ^{-1}\colon L^2 (\R) \to L^2 (\R)
\end{align}
which can be extended to a meromorphic family of operators 
\begin{align}
    R_V (\lambda) \coloneq (D_x^2 + V(x) -\lambda^2)^{-1} \colon L^2_{\rm{c}} (\R) \to L^2_{\rm{loc}} (\R)
\end{align}
for all $\lambda \in \C$.
Moreover, if $u_\pm(x,\lambda)$ are outgoing solutions, as constructed in Proposition \ref{prop:existenceofoutgoing}, then the poles of the resolvent are the $\lambda \in \C$ such that $W[u_+,u_-](\lambda)=0$ (defined in \eqref{eq:wronk}), which are notably independent of $x$.
\end{prop}
\begin{proof}
If $f \in L^2(\R)$ and $\Im \lambda \gg 1$, then using the variation of parameters formula, a solution to
\begin{align}
    (D_x^2 + V(x) - \lambda^2 ) u = f
\end{align}
can be written as
\begin{align}
    u(x) = \int_\R R(x,y,\lambda) f(y) \dd y
\end{align}
where:
\begin{align}
    R(x,y,\lambda) \coloneq  \frac{u_+(x,\lambda) u_-(y,\lambda) H(x-y)}{W[u_+,u_-](\lambda)} + \frac{u_-(x,\lambda) u_+(y,\lambda) H(y-x)}{W[u_+,u_-](\lambda)}.\label{eq:535}
\end{align}

Here, $u_\pm$ are solutions to $ (D_x^2 + V(x) - \lambda^2 ) u_\pm =0 $ as constructed in Proposition \ref{prop:existenceofoutgoing}, and $H(x)$ is the Heaviside function. 
By \Cref{prop:Wronskian}, the Wronskian for $u_{\pm}$ is not identically zero. Moreover, because the Wronskian is entire in $\lambda$, it follows that the resolvent can be extended to a meromorphic function of $\lambda$.
 {Because $u_\pm$ do not identically vanish for all $\lambda$, the poles of $R_V(\lambda)$ are exactly where the Wronskian is zero.}
\end{proof}

\subsection{The Scattering Matrix}
Because $V$ decays exponentially, we have the existence of the scattering matrix $S(\lambda)$ \cite[\S 5]{melin1985operator}, whose components we write as
\begin{align}
    S(\lambda) \coloneqq \mat{T(\lambda)  & R_+(\lambda) \\ R_-(\lambda)  & T(\lambda)}\label{eq:645}
\end{align}
where $T(\lambda)$ is the transmission coefficient, $R_+(\lambda)$ is the right reflection coefficient, and $R_-(\lambda)$ is the left reflection coefficient.
{We recall that for compactly supported solutions to $ ( P_V - \lambda^2 ) u = 0$, $ \lambda \in \mathbb R \setminus \{ 0 \} $ have the form
\[   u|_{ x \ll - 1 } = A_- e^{ i \lambda x } + B_- e^{ - i \lambda x} , \ \ 
 u|_{ x \gg 1 } = A_+ e^{ i \lambda x } + B_+ e^{ - i \lambda x} , \]
 and $ S ( \lambda ) $ takes the incoming terms to the outgoing terms, with the 
 convention that for $ V = 0 $ the scattering matrix is the identity:
 \[  S( \lambda ) \begin{pmatrix}  A_- \\ B_+ \end{pmatrix} = \begin{pmatrix}  A_+ \\ B_- \end{pmatrix} .
 \]
In our situation the same definition applies but now in an asymptotic sense as
$ x \to \pm \infty $.}

Let $u_\pm = u_\pm(x,\lambda)$ be outgoing solutions at $\pm \infty$ and let $w_\pm = w_\pm (x,\lambda)$ be incoming solutions at $\pm \infty$ (recalling Definitions \ref{def:outgoing} and \ref{def:incoming}) so that 
\begin{align}
    u_\pm (x)|_{\pm x \gg 1} = e^{\pm i\lambda x} v_\pm (e^{\mp A_\pm x}) , \ \  w_\pm (x)|_{\pm x \gg 1} = e^{\mp i\lambda x} \tilde v_\pm (e^{\mp  A_\pm x}). 
\end{align}

For now, we will suppress the $\lambda$ in the arguments of $u_\pm$ and $w_\pm$.
Because $u_+$ and $w_+$ are linearly independent, there exist constants $a,b\in \C$ (depending on $\lambda$) such that
\begin{align}
    u_-(x) = a u_+(x) + bw_+(x). \label{eq:613}
\end{align}

The scattering matrix will then satisfy
\begin{align}
    S(\lambda) \mat{0 \\ b\tilde v_+(0)} = \mat{{ }av_+(0)\\v_- (0)}\label{eq:661}
\end{align}
so that
\begin{align}
    T(\lambda) b\tilde v_+(0) = v_-(0).\label{eq:625}
\end{align}

The transmission coefficient $T(\lambda)$ can be related to the Wronskian of $u_+$ and $u_-$.
Indeed, using \eqref{eq:613},
\begin{align}
    W[u_+,u_-](\lambda) &= b W[u_+,w_+] (\lambda)= b \lim_{x\to \infty} W[e^{i\lambda x} v_+ (e^{-A_+ x}),e^{-i\lambda x }\tilde v_+(e^{-A_+ x})]\\
    &= b2i\lambda v_+(0) \tilde v_+(0)= \frac{2i \lambda v_+(0)v_-(0)}{T(\lambda)}\label{eq:672}
\end{align}
where in the last equality we used \eqref{eq:625}.

Finally using \Cref{prop:existenceofoutgoing}, we can rewrite \eqref{eq:672} as  {
\begin{align}
    T(\lambda) = \frac{2i\lambda }{\Gamma_{\Lambda_+}(1- 2i \lambda A_+^{-1} )\Gamma_{\Lambda_-}(1- 2i \lambda A_-^{-1})W[u_+,u_-](\lambda)}\label{eq:598}
\end{align}
}
which is a meromorphic function and has poles only at the zeros (in $\lambda$) of $W[u_+,u_-](\lambda)$. 

\begin{rem}
    Observe that, besides $\lambda = 0$, the only zeros of $T(\lambda)$ are at the poles of  {$\Gamma_{\Lambda_+}(1- 2i \lambda A_+^{-1} )$ and $\Gamma_{\Lambda_-}(1- 2i \lambda A_-^{-1})$} where $W[u_+,u_-](\lambda)\neq 0$, which are along the negative imaginary axis at the points
 {
\begin{align}
    \set{-\frac{1}{2} A_\pm i k : k  \in \Z_{\ge 1} \setminus \Lambda_{\pm}}.
\end{align}
}
By unitarity of the scattering matrix, we see that these points provide ``false'' poles of the scattering matrix along the positive imaginary axis at the points:
 {
\begin{align}
    \set{\frac{1}{2} A_\pm i k : k  \in \Z_{\ge 1} \setminus \Lambda_{\pm}}\setminus \set{\lambda \in \C : W[u_+,u_-](-\lambda ) = 0}.
\end{align}
}
The existence of such poles contradicts Heisenberg's original hypothesis that all poles of the scattering matrix correspond to resonances.\footnote{See for instance \cite[\S 12.1.2]{newton2013scattering} for a discussion of false poles in the context of Jost functions. See also \cite{Pais1995} for a bibliographical account of Jost which includes a brief discussion, with references, of ``false'' poles.
As explained by Graham--Zworski \cite{graham2003scattering}, these ``false'' poles also play an important role in scattering on asymptotically hyperbolic spaces and its relation to the conformal structure of the boundary at infinity.
}
These ``false'' poles will give rise to the $\mathcal{A}_\pm (t)$ appearing in our trace formula \eqref{eq:218}.
\end{rem}

\section{Applying the Birman--Kre\u{\i}n trace formula}\label{s:birman}

We are now prepared to apply the Birman--Kre\u{\i}n trace formula to prove our main result.
Recall we are computing $\Tr(U_V(t) - U_0(t))$, where $U_V(t) \coloneq \cos (t\sqrt{P_V})$ is the wave propagator for a potential $V(x)$.
We fix $\phi \in C_0^\infty ((0,\infty))$\footnote{There is a singularity at $t=0$ which requires us to choose $\phi$ supported in positive time.}, so that
\begin{align}\label{traceapp}
    \Tr(U_V(t) - U_0(t)) (\phi) &=  \Tr(\cos(t\sqrt{P_V}) - \cos(t\sqrt{P_0}) )(\phi) \\
    &= \Tr(f(P_V) - f(P_0) ) \label{eq:763}
\end{align}
where $f(\lambda^2 ) = \frac{1}{2} (\hat \phi (\lambda) + \hat \phi (-\lambda))$. Equation \eqref{eq:763} follows by writing
\begin{align*}
    \cos\left(t \sqrt{P_V}\right ) (\phi) &= \tfrac{1}{2}\int_{-\infty}^\infty \left( e^{it \sqrt{P_V} } +  e^{-it\sqrt{P_V}}\right) \phi(t) \dd t\\
    &= \tfrac{1}{2} (\hat \phi (\sqrt{P_V}) + \hat \phi(-\sqrt{P_V})) = f(P_V).
\end{align*}

By the Birman--Kre\u{\i}n trace formula (\Cref{t:BK}), we obtain
\begin{align}
   \begin{split}
        \Tr(f(P_V) - f(P_0)) = \frac{1}{2\pi i } &\int_0^\infty f(\lambda^2 ) \Tr (S(\lambda)^{-1} \p_\lambda S(\lambda )) \dd \lambda \\
        &+ \sum_{j=1}^k f(E_j) + \tfrac{1}{2} (m_R(0) - 1) f(0)
   \end{split}  \label{BK1}
\end{align}
where $E_j$ are the negative eigenvalues of $D_x^2 + V(x)$. 
Note that if $V$ is positive (which we do not necessarily assume here), there are no negative eigenvalues.

We now focus on the first term on the right-hand side of \eqref{BK1}.
The term we are integrating against $f(\lambda)$ is
\begin{align}
  \mathcal{G}(\lambda)\coloneq  \Tr \left ( S(\lambda)^{-1} \p_\lambda S(\lambda ) \right) & = \frac{d}{d\lambda} \log \det (S(\lambda))=\frac{d}{d\lambda} \log \left(  T(\lambda)/T(-\lambda)\right)\label{eq:682}
\end{align}
where the first equality follows from Jacobi's formula and the second is a standard fact about the scattering matrix (see, for instance, \cite[\S 2.10 Exercise 3]{dyatlov2019mathematical}).
Moreover, we have by \eqref{eq:598}
 {
\begin{align*}
    T(\lambda) = \frac{2i\lambda}{\Gamma_{\Lambda_-}(1-\alpha_-)\Gamma_{\Lambda_+}(1-\alpha_+)F(\lambda)}
\end{align*}
}
where 
\begin{align}
    F(\lambda) \coloneq W[u_+ (\cdot ,\lambda ) ,u_-(\cdot ,\lambda)]\label{def:F}
\end{align}
is entire in $\lambda$, with $u_+$ and $u_-$ constructed in \Cref{prop:existenceofoutgoing} and $\alpha_\pm = 2i\lambda A_{\pm}^{-1}$.
Therefore, we can expand
 {
\begin{equation} \label{eq:740}
\begin{split}
    \log(T(\lambda) / T(-\lambda) ) &= \log(\Gamma_{\Lambda_-}(1+\alpha_-)) - \log(\Gamma_{\Lambda_-}(1 - \alpha_-))+\log(\Gamma_{\Lambda_+}(1+\alpha_+)) \\
    & \qquad \qquad- \log(\Gamma_{\Lambda_+}(1-\alpha_+))    + \log(F(-\lambda)) - \log (F(\lambda)).
    \end{split}
\end{equation}
}
Computing the derivative, we obtain
\begin{equation}
\mathcal{G}(\lambda) = \frac{d}{d\lambda}\log(T(\lambda)/T(-\lambda))(\lambda)=\mathcal{G}_{\Gamma}^+(\lambda)+\mathcal{G}_{\Gamma}^-(\lambda)+\mathcal{G}_F(\lambda) \label{eq:887}
\end{equation}
where we define
 {
\begin{align}
    \mathcal{G}_\Gamma^\pm(\lambda)  & \coloneq \frac{d}{d\lambda} \left( \log(\Gamma_{\Lambda_\pm}(1+\alpha_\pm)) - \log(\Gamma_{\Lambda_\pm}(1 - \alpha_\pm))\right)\\
    &= \frac{2iA_\pm^{-1} \Gamma_{\Lambda_\pm}'(1 + \alpha_\pm)}{\Gamma_{\Lambda_\pm}(1 + \alpha_\pm)} + \frac{2iA_\pm^{-1}\Gamma_{\Lambda_\pm}'(1-\alpha_\pm)}{\Gamma_{\Lambda_\pm}(1-\alpha_\pm)} \label{eqa2}
\end{align}
}
and
\begin{align}
\mathcal{G}_F(\lambda)\coloneq \frac{d}{d\lambda}(\log(F(-\lambda)) - \log (F(\lambda))). \label{eq:751}
\end{align}
Inserting this into the first term on the right-hand side of \eqref{BK1} gives
\begin{equation}\label{firsttermcontributions}
\begin{split}
    &\frac{1}{2\pi i}\int_{0}^{\infty}f(\lambda^2)\Tr(S(\lambda)^{-1}\partial_{\lambda}S(\lambda))\dd \lambda\\
   =&\frac{1}{2\pi i}\int_{0}^{\infty}f(\lambda^2)\left(\mathcal{G}_{F}(\lambda)+\mathcal{G}_{\Gamma}^+(\lambda)+\mathcal{G}_{\Gamma}^-(\lambda)\right)\dd \lambda
\end{split}
\end{equation}
It will turn out to be relatively straightforward to analyze the contribution of $\mathcal{G}_{\Gamma}^\pm$ in \eqref{firsttermcontributions} using  {an analogue of} the well-known series expansion of the digamma function. 
On the other hand, to analyze $\mathcal{G}_F$, we will need some precise quantitative information about the distribution of the zeros of $F$.
Our starting point is to observe that, thanks to \eqref{eq:boundonvp} and the fact that the Wronskian is constant in $x$, we have the upper bound
\begin{equation*}
|F(\lambda)|\lesssim e^{C|\lambda|\log(1+|\lambda|)}.
\end{equation*}
Since $F$ is entire in $\lambda$, this yields the Hadamard factorization (cf. \cite[Theorem 1.9]{Hayman1964})
\begin{align}
    F(\lambda) = e^{a\lambda + b  }\lambda^m \prod_{j=1}^\infty\left( 1 - \frac{\lambda }{\lambda_j} \right) e^{\lambda / \lambda_j}.
\end{align}
for some $a,b\in\mathbb{C}$. Here $m$ is the multiplicity of the zero at $0$ and the $\lambda_j$ are the nonzero roots of $F$. We then observe the following (formal at this point) expansion of $\log(F(\lambda))$,
\begin{align}
    \log(F(\lambda)) = (a\lambda + b) + m \log (\lambda) + \sum_{j=1}^\infty \left( \log\left(1 - \frac{\lambda}{\lambda_j} \right)  + \frac{\lambda}{\lambda_j} \right)
\end{align}
so that
\begin{align}
    \frac{d}{d\lambda}  \log(F(\lambda))  &= a + \frac{m}{\lambda} + \sum_{j=1}^\infty \left( \frac{1}{\lambda-\lambda_j} + \frac{1}{\lambda_j}  \right). 
\end{align}
Similarly, we have
\begin{align}
      \frac{d}{d\lambda}  \log(F(-\lambda))&= -a + \frac{m}{\lambda} + \sum_{j=1}^\infty \left( \frac{1}{\lambda_j +\lambda}   - \frac{1}{\lambda_j}\right).
\end{align} 
Therefore, we have the following formal expansion of $\mathcal{G}_F$ (defined in \eqref{eq:751}),
\begin{align}
    \mathcal{G}_F(\lambda) = -2a + \sum_{j=1}^\infty \left( \frac{1}{\lambda_j +\lambda} + \frac{1}{\lambda_j - \lambda}    - \frac{2}{\lambda_j}\right). \label{AFdef}
\end{align}
We observe in particular that $\mathcal{G}_F$ is an even function in $\lambda$. In order to compute the contribution of this term in \eqref{firsttermcontributions} and rigorously justify the forthcoming manipulations, we will need the following proposition, which gives a precise description of the location and density of the zeros of $F$.
\begin{prop}\label{Fprop}
The roots of $F(\lambda) = W[u_+,u_-](\lambda)$ satisfy the following properties
\begin{enumerate}
    \item \label{F:prop1} There exists $C>0$ such that:
    \begin{align}
        \set{F^{-1}(0)} \subset \set{ i y : y\in [0,C]} \cup \set{z \in \C : \Im{z} < 0}.
    \end{align}
    \item \label{F:prop2} If $F(\lambda) = 0$, then $F(-\bar\lambda)= 0$.\footnote{In fact for all $z\in \C$, $F(z) = \overline{F(-\bar z)}.$} 
    \item \label{F:prop3} There exists $\delta>0$ such that $F$ does not vanish anywhere on the strip $S_{\delta}\coloneq\{\lambda \in \C : -\delta<\operatorname{Im}(\lambda)<0\}$.\footnote{We remark that \cite{martinez2002resonance} can be used to show that $F$ does not vanish on a logarithmic region, however this is not needed in this article.}
    \item \label{F:prop4} For each $\e>0$, there exists a $C>0$ such that for all $r >0$
    \begin{align}
        \#\set{z \in \C : F(z) =0 , |z| \le r}\le C (1+r^{1+\e}).
    \end{align}
    
\end{enumerate}
\end{prop}
For ease of exposition, we postpone the proof of this proposition until the end of the section.

We now compute the contribution of $\mathcal{G}_{F}$ in \eqref{firsttermcontributions}. Using that $\mathcal{G}_{F}$ is even, we observe the expansion
\begin{align}
\frac{1}{2\pi i} \int_0^\infty f(\lambda^2 )\mathcal{G}_{F}(\lambda)\dd \lambda&= \frac{1}{4\pi i} \int_0^\infty (\hat \phi(\lambda) + \hat \phi (-\lambda)) \mathcal{G}_F(\lambda)\dd \lambda\\
&= \frac{1}{4\pi i} \int_{-\infty}^\infty \hat \phi (\lambda)\mathcal{G}_F(\lambda)\dd \lambda \label{eq:619}
\end{align}
To compute the integral in the second line, our first objective will be to establish the uniform summability bound
\begin{lemma}\label{summabilityestimate}
Let $\mathcal{G}_F^j$ denote the $j$th summand in \eqref{AFdef}. Then, for every $\epsilon>0$ sufficiently small, we have the uniform bound
\begin{equation}\label{unifbound}
\sum_{j=1}^{\infty}|\mathcal{G}_F^j(\lambda)|\leq C_{\epsilon}\langle\lambda\rangle^{1+\epsilon},\hspace{5mm}\lambda\in\mathbb{R}
\end{equation}
with $C_\e > 0$.
\end{lemma}
The above lemma asserts, in particular that for each $\lambda$, $\mathcal{G}_F^j$ is absolutely summable with a sub-polynomial upper bound in $\lambda\in\mathbb{R}$. We note importantly that this bound holds for real $\lambda$. 
\begin{proof}
Without loss of generality, we assume $\lambda_j$ are indexed in ascending order of magnitude. 
That is,
\begin{equation*}
|\lambda_1|\leq |\lambda_2|\leq...\leq|\lambda_k|\leq\cdots.
\end{equation*}
Since $F$ is entire and $\lambda_j\neq 0$, we have the lower bound 
\begin{equation}\label{lowerlambdabound}
\inf_{j\in\mathbb{N}}|\lambda_j|>c_0
\end{equation}
for some $c_0>0$ depending on the profile of $F$. 
Moreover by Property \ref{F:prop4} in \Cref{Fprop}, we have
\begin{equation}\label{lambdajbound}
|\lambda_j|\geq Cj^{1-\epsilon}
\end{equation}
where $C>0$ depends on $\epsilon$.\footnote{
Indeed, if \eqref{lambdajbound} were not true, then for every $\tilde C >0$, there would exist a subsequence $j_k \to \infty$ such that $\abs{\lambda_{j_k}}  < \tilde C (j_k )^{1-\e}$.
But because $\lambda_j$ are ordered by absolute value, we have that $j_k  \le \# \set{|\lambda_{j}| \le (j_k)^{1-\e} \tilde C }\le C(1 + ((j_k)^{1-\e} \tilde C)^{1+\e})$.
This implies that $j_k \le C (1 + (j_k)^{1-\e^2} \tilde C ^{1+\e})$ which cannot hold as $k \to \infty$.
} 
As above, in the sequel, we will adopt the convention that $C>0$ denotes a generic constant (which can change from line to line) depending on $\epsilon$ and on $c_0$, but is otherwise universal. 
Now we consider separately the regions $|\lambda_j|> 2|\lambda|$ and $|\lambda_j|\leq 2|\lambda|$. When $|\lambda_j|\geq 2|\lambda|$, we have by straightforward algebraic manipulation and \eqref{lambdajbound}, 
\begin{equation*}
\begin{split}
\abs{\frac{1}{\lambda_j+\lambda}+\frac{1}{\lambda_j-\lambda}-\frac{2}{\lambda_j}}&=\abs{\frac{\lambda}{\lambda_j(\lambda_j-\lambda)}-\frac{\lambda}{\lambda_j(\lambda_j+\lambda)}}
\\
&\leq \abs{\frac{\lambda}{\lambda_j(\lambda_j-\lambda)}}+\abs{\frac{\lambda}{\lambda_j(\lambda_j+\lambda)}}
\\
&\leq C\frac{|\lambda|}{j^{2(1-\epsilon)}}
\end{split}
\end{equation*}
Above, in the last estimate, we used that $|\lambda_j+\lambda|\approx |\lambda_j-\lambda|\approx |\lambda_j|$. Hence,
\begin{equation*}
\sum_{|\lambda_j|\geq 2|\lambda|}|\mathcal{G}_F^j(\lambda)|\leq C|\lambda|.
\end{equation*}
In the region $|\lambda_j|\leq 2|\lambda|$, we may assume $|\lambda|\geq c_0 / 2$, otherwise this region is empty and we are done. We begin by showing the bound
\begin{equation}\label{summandbound}
\frac{1}{|\lambda_j\pm\lambda|}+\frac{1}{|\lambda_j|}\leq C.
\end{equation}
We already have the required bound for $\abs{\lambda_j}^{-1}$ in view of (\ref{lowerlambdabound}). To estimate $|\lambda_j\pm\lambda|$, we have for some $\delta>0$ (in view of the fact that $\lambda$ is real),
\begin{equation*}
|\lambda_j\pm\lambda|\geq |\operatorname{Im}(\lambda_j)|\geq \delta.
\end{equation*}
In the second inequality, we used Property \ref{F:prop3} of \Cref{Fprop}. Therefore, we have \eqref{summandbound}, and thus (in light of Property \ref{F:prop4} of \Cref{Fprop}), we have
\begin{equation*}
\sum_{|\lambda_j|\leq 2|\lambda|}|\mathcal{G}_F^j(\lambda)|\leq C\#\set{j: |\lambda_j|\leq 2|\lambda|}\leq C(1+|\lambda|^{1+\epsilon})\leq C\langle\lambda\rangle^{1+\epsilon}.
\end{equation*}
Combining everything yields (\ref{unifbound}), as desired.
\end{proof}
Now, we return to our analysis of \eqref{eq:619}. Using \Cref{summabilityestimate}, that $\hat{\varphi}$ is Schwartz, $\varphi(0)=0$, and Fubini's Theorem, we compute
\begin{equation*}
\begin{split}
\frac{1}{4\pi i } \int_{-\infty}^\infty \hat{\varphi}(\lambda)\mathcal{G}_F(\lambda)\dd \lambda&=\frac{1}{4\pi i}\sum_{j=1}^{\infty}\int_{-\infty}^{\infty}\hat{\varphi}(\lambda)\mathcal{G}_F^j(\lambda)\dd \lambda.
\end{split}
\end{equation*}
We note that the term $-2a$ from \eqref{AFdef} does not contribute to the sum as $\int_{-\infty}^\infty \hat \phi (\lambda ) \dd \lambda = 0$.
It remains to compute each of the above summands. For this, we will carry out a simple contour deformation. We let $R_n\to\infty$ be a sequence to be chosen. Clearly we have for each $j$
\begin{equation*}
\int_{-\infty}^{\infty}\hat{\varphi}(\lambda)\mathcal{G}_F^j(\lambda)\dd \lambda=\int_{-R_n}^{R_n}\hat{\varphi}(\lambda)\mathcal{G}_F^j(\lambda)\dd \lambda+O(R_n^{-\infty}).
\end{equation*}
 Let $\gamma_1^n$ be the line from $-R_n$ to $-R_n-iR_n$, $\gamma_2^n$ be the line from $-R_n-iR_n$ to $R_n-iR_n$, and $\gamma_3^n$ be the line from $R_n-iR_n$ to $R_n$. We denote by $\gamma^n$, the union of these three lines oriented clockwise and let $D_n$ denote the rectangle enclosed by $\gamma^n$ and the real axis. Choosing $R_n$ such that the contour does not intersect any of the poles of $\mathcal{G}_F$, we have by the residue theorem,
\begin{equation}\label{contourdef2}
\int_{-R_n}^{R_n}\hat{\varphi}(\lambda)\mathcal{G}_F^j(\lambda)\dd\lambda=-\int_{\gamma^n}\hat{\varphi}(\lambda)\mathcal{G}_F^j(\lambda)\dd \lambda-2\pi i \1_{ D_n}(\lambda_j) \text{Res}(\hat{\varphi}(\lambda) \mathcal{G}_F^j(\lambda),\lambda_j).
\end{equation}
Here $\1_{D_n}(\lambda_j)$ is $1$ if $\lambda_j \in D_n$ and $0$ otherwise.
By Property \ref{F:prop4} of \Cref{Fprop} and pigeonholing, we can further arrange for the sequence $R_n$ to satisfying the following properties:
\begin{enumerate}
\item (Approximate unit spacing). The sequence $R_n$ satisfies
\begin{equation}\label{Rnapprox}
|R_n-n|<1.
\end{equation}
\item (Quantitative avoidance of poles) For $\lambda \in \gamma^n$, we have
\begin{equation}\label{lowerbound}
\inf_{j}|\lambda_j\pm\lambda|\geq Cn^{-2}.
\end{equation}
\end{enumerate}
\begin{rem}
$n^{-2}$ is not optimal but will suffice for our purposes.
\end{rem}
It remains to estimate the contributions of $\gamma_i^n$ for each $i=1,2,3$. First, for $\lambda$ in $\gamma_1^n$ or $\gamma_3^n$, we compute from \eqref{lowerbound} and the fact that the length of $\gamma^n$ is on the order of $n$,
\begin{equation*}
\int_{\gamma^n}|\hat{\varphi}(\lambda)||\mathcal{G}_F^j(\lambda)|\dd \lambda\leq Cn^3\sup_{\lambda\in \gamma^n}|\hat{\varphi}(\lambda)|.
\end{equation*}
In view of the support properties of $\varphi$ and that $\lambda$ is in the lower half plane, we have for $\lambda \in \gamma _1^n\cup \gamma_3^n$ 
\begin{equation*}
n^4|\hat{\varphi}(\lambda)|\leq C|\lambda|^4|\hat{\varphi}(\lambda)|\leq C\int_{0}^{\infty}|\varphi^{(4)}(t)|e^{t\operatorname{Im}(\lambda)}\dd t\leq C\|\varphi^{(4)}\|_{L^1(\mathbb{R})}.
\end{equation*} 
For $\lambda$ on the contour $\gamma_2^n$, we use the fact that $\Im \lambda =-R_n$ to obtain
\begin{equation*}
|\hat{\varphi}(\lambda)|\leq Ce^{-R_n}\|\varphi\|_{L^1(\mathbb{R})}.
\end{equation*}
Combining the above, we obtain
\begin{equation}\label{vanishingcontour}
\int_{\gamma^n}\hat{\varphi}(\lambda)\mathcal{G}_F^j(\lambda)\dd\lambda\to 0,\hspace{5mm}\text{as}\hspace{2mm}n\to\infty.
\end{equation}
It remains to compute the sum of the second term on the right-hand side of \eqref{contourdef2} over $j$. Indeed, we have
\begin{align}
    \sum_{j=1}^\infty \int_{-\infty}^\infty \hat \phi (\lambda ) \mathcal{G}_F^j (\lambda ) \dd \lambda &= \sum_{j=1}^\infty \lim_{n\to \infty} \int_{-R_n}^{R_n} \hat \phi (\lambda) \mathcal G_F^j (\lambda ) \dd \lambda\\
&= -2\pi i \sum_{j=1}^\infty\lim_{n\to\infty} \1_{D_n} (\lambda_j)\mathrm{Res}(\hat{\varphi}(\lambda) \mathcal{G}_F^j(\lambda),\lambda_j)\\
&= -2\pi i  \left( \sum_{\Im{\lambda_j}>0}\hat \phi (-\lambda _j)- \sum_{\Im{\lambda _j} <0} \hat \phi (\lambda _j) \right)\label{eq:1119}
\end{align}

From the definition of $f$ and the fact that the $\lambda_j$ with positive imaginary part satisfy $\operatorname{Re}(\lambda_j)=0$ (by \Cref{Fprop} Property \eqref{F:prop1}), we can further write
\begin{equation}
\sum_{\operatorname{Im(\lambda_j)}>0}\hat{\varphi}(-\lambda_j)=2\sum_{\Im {\lambda_j } > 0}f(\lambda_j^2)-\sum_{\operatorname{Im}(\lambda_j)>0}\hat{\varphi}(\lambda_j) \label{eq:1130}
\end{equation}
Let $\tilde E_j = \lambda_j^2$ for $\Im{\lambda_j}>0$. 
By \Cref{Fprop}, there are finitely many $\tilde E_j$ (whose cardinality we denote by $k'$) which are negative real numbers (we will later show they are the negative eigenvalues of $P_V$).
We can now rewrite \eqref{eq:1119} using \eqref{eq:1130} as
\begin{equation*}
\sum_{j=1}^\infty \int_{-\infty}^\infty \hat \phi (\lambda ) \mathcal{G}_F^j (\lambda ) \dd \lambda =-2\pi i\left( 2 \sum_{j=1}^{k'} f(\tilde E_j) - \sum_{j=1}^\infty \hat \phi (\lambda _j) \right)
\end{equation*}

Thanks to \eqref{contourdef2}, we can insert this into \eqref{BK1} to obtain 
\begin{align}
   \begin{split}
        \tr(f(P_V) - f(P_0)) =&\frac{1}{2}\sum_{j=1}^{\infty}\hat{\varphi}(\lambda_j)+\frac{1}{4\pi i } \int_{-\infty}^\infty \hat{\varphi}(\lambda)(\mathcal{G}_{\Gamma}^+(\lambda)+\mathcal{G}_{\Gamma}^-(\lambda))\dd \lambda \\
        &+ \frac{1}{2} (m_R(0) - 1) f(0) +    \sum_{j=1}^k f(E_j) - \sum_{j=1}^{k'} f(\tilde E_j).
   \end{split}  \label{BK2}
\end{align}
\
It remains to compute the contributions of $\mathcal{G}_{\Gamma}^+$ and $\mathcal{G}_{\Gamma}^-$. We recall from \eqref{eqa2} the formula,
 {
\begin{equation*}
\begin{split}
 \mathcal{G}_\Gamma^\pm(\lambda) = \frac{2iA_\pm^{-1} \Gamma_{\Lambda_\pm}'(1 + \alpha_\pm)}{\Gamma_{\Lambda_\pm}(1 + \alpha_\pm)} + \frac{2iA_\pm^{-1}\Gamma_{\Lambda_\pm}'(1-\alpha_\pm))}{\Gamma_{\Lambda_\pm}(1-\alpha_\pm)}
\end{split}
\end{equation*}
}
where $\alpha_\pm=2i\lambda A^{-1}_\pm$. The main ingredient here  {is a simple analogue of} the the following well-known formula for the digamma function,
\begin{align}
    \Psi(z) \coloneq \frac{\Gamma'(z)}{\Gamma(z)} = -\gamma + \sum_{j=0}^\infty \left ( \frac{1}{j+1} - \frac{1}{j+z} \right),
\end{align}
where $\gamma$ is the Euler–Mascheroni constant.  {A simple modification of the above formula shows that the logarithmic derivative of our modified Gamma functions $\Gamma_{\Lambda_\pm}$ is given by 
\begin{equation*}
\Psi_{\Lambda_\pm}(z)\coloneq \frac{\Gamma_{\Lambda_\pm}'(z)}{\Gamma_{\Lambda_\pm}(z)}=-\gamma+\sum_{j\in \Z_{\ge0}\setminus{\tilde{\Lambda}_{\pm}}}\left(\frac{1}{j+1}-\frac{1}{j+z}\right)    
\end{equation*}
where $\tilde{\Lambda}_{\pm}\coloneq \Lambda_\pm-1\subset \Z_{\ge 1}$.
Using this expansion, one can carry out a contour deformation argument as in the analysis of the term $\mathcal{G}_F(\lambda)$ to write  
\begin{align}
    &\frac{1}{2\pi i}\int_0^\infty f(\lambda^2 ) (\mathcal G_\Gamma ^+ (\lambda)+\mathcal{G}_\Gamma ^-(\lambda)) \dd \lambda \\&= \tfrac{1}{4\pi i }\int_{-\infty}^\infty \hat{\varphi}(\lambda)(\mathcal{G}_{\Gamma}^+(\lambda)+\mathcal{G}_{\Gamma}^-(\lambda))\dd \lambda\\
    &=-\tfrac{1}{2}\sum_{j\in\Z_{\ge 0}\setminus{\tilde{\Lambda}_+}}\hat{\varphi}\left(-iA_+\frac{j+1}{2}\right)-\tfrac{1}{2}\sum_{j\in\mathbb{Z}_{\ge 0}\setminus{\tilde{\Lambda}_-}}\hat{\varphi}\left(-iA_-\frac{j+1}{2}\right).
\end{align}
Writing out the definition of $\hat{\varphi}$ and using Fubini, we can rewrite this as
\begin{equation*}
\int_{-\infty}^{\infty}\varphi(t)\left(\mathcal{A}_+(t)+\mathcal{A}_-(t)\right)\dd t
\end{equation*}
where
\begin{equation*}
\mathcal{A}_\pm(t)\coloneq -\tfrac{1}{2}\sum_{j\in\mathbb{N}_0\setminus{\tilde{\Lambda}_\pm}}e^{-tA_\pm\frac{(j+1)}{2}}=-\tfrac{1}{2}\sum_{j\in\mathbb{Z}_{\geq 1}\setminus{\Lambda_\pm}}e^{-tA_\pm\frac{j}{2}}.    
\end{equation*}
which matches the general form of $\mathcal{A}_\pm$ in equation \eqref{eq:450}. We also note that if $\Lambda_+$ and $\Lambda_-$ are both empty, or in other words, $V\notin\mathcal{R}\mathcal{C}$, then we can write the more explicit formula for $\mathcal{A}_\pm(t)$,
\begin{align}
\mathcal{A}_{\pm}(t)&=-\tfrac{1}{2}\sum_{j=0}^\infty e^{-tA_\pm/2 } e^{-tA_\pm j/2}
=-\tfrac{1}{2}\frac{ e^{-tA_\pm/2 }}{ 1 - e^{-tA_\pm/2}} =-\tfrac{1}{2}\frac{ 1}{e^{tA_\pm /2}-1} 
\end{align}
which agrees with equation \eqref{eq:209}. In general, combining the above with \eqref{BK2} and \eqref{traceapp} and using that $\varphi$ is supported on $(0,\infty)$, we obtain
\begin{equation}\label{finaltrace}
\begin{split}
 \Tr(U_V(t) - U_0(t)) &=\tfrac{1}{2}\sum_{j=1}^{\infty}e^{-i\lambda_jt}+\mathcal{A}_+(t)+\mathcal{A}_-(t)\\
        &+ \tfrac{1}{2} (m_R(0) - 1) + \sum_{j=1}^k f(E_j) - \sum_{j=1}^{k'} f(\tilde E_j)
\end{split}
\end{equation}
}
in the sense of distributions. To conclude the proof of \label{eq:507}, we need to show that the $\lambda_j$ appearing in the first term on the right-hand side correspond to non-zero resonances of $V$ and that the $\tilde{E}_j$ correspond to the negative eigenvalues for $V$.
This is the content of the following lemma, which is proved in Appendix \ref{s:proof of lemmainduct}.
\begin{lemma}\label{lemma:inductthing}
We have that $\lambda_0$ is a pole of $R_V(\lambda)$ with multiplicity $ m $ if and only if $\lambda_0$ is an order $m$ zero of $F(\lambda) = W[u_+(\cdot,\lambda),u_-(\cdot ,\lambda)]$. 
Moreover, if $0$ is a resonance (see \Cref{def:multi}) of $-\p_x^2 + V(x)$, then it is simple.
\end{lemma}
With this, we see that $\tilde E_j$ are the (negative) eigenvalues (corresponding to the poles of $R_V(\lambda)$ in the upper half plane) of $P_V$ so the sum of the last two sums in \eqref{finaltrace} are zero, and 
\begin{align}
    \tfrac{1}{2} \sum_{j=0}^\infty e^{-i\lambda _j t} + \tfrac{1}{2} (m_R(0) -1) = \tfrac{1}{2} \sum_{\lambda_j \in \Res(V)} e^{-i\lambda _j t} - \tfrac{1}{2}
\end{align}
which establishes the trace formula \eqref{eq:218}. It remains to finally prove \Cref{Fprop}, which we originally postponed. 

\begin{proof}[Proof of \Cref{Fprop}]\phantom{ }
\noindent Properties \ref{F:prop1} and \ref{F:prop2} are standard results in scattering theory for real-valued potentials (see, for instance, \cite[\S 2.2]{dyatlov2019mathematical}).
Indeed, the roots of $F$ correspond to the resonances of $P_V$ by \Cref{lemma:inductthing}. 
Because the spectrum of $P_V$ is contained in the positive real axis with possibly finitely many negative eigenvalues, the resolvent $(P_V -\lambda^2)$ is a meromorphic operator for $\Im \lambda >0$ on $L^2$ with finitely many poles along the imaginary axis corresponding to negative eigenvalues of $P_V$ (this proves \ref{F:prop1}).
Property \ref{F:prop2} follows by taking the complex conjugate of outgoing solutions to $P_V -\lambda ^2 $.

 We next prove Property \ref{F:prop3}. 
 In view of Property \ref{F:prop1}, it suffices to show that there is some $\delta>0$ such that $F$ does not vanish anywhere on $S_{M,\delta}\coloneq S_{\delta}\cap \{\lambda: |\lambda|\geq M\}$ where $\delta>0$ and $M>0$ are some positive constants to be chosen. Moreover, thanks to \Cref{lemma:inductthing}, it is sufficient to show that $S_{M,\delta}$ is resonance-free. In view of this, we need to show that for every $\rho\in C_c^{\infty}(\mathbb{R})$ and $\lambda\in S_{M,\delta}$, there holds
\begin{equation}\label{L2L2bound}
\|\rho R_V\rho\|_{L^2\to L^2}<\infty.
\end{equation}

To set the stage, we use the standard formula for the resolvent, see for instance the discussion before \cite[Lemma 3.1]{froese1997asymptotic}. 
For that we define
\[  V^{\frac12} \coloneq \text{sgn}( V ) |V|^{\frac12}, \ \ \ \  \mathbf R_V ( \lambda ) \coloneq V^{\frac12} R_0 ( \lambda ) 
|V|^{\frac12}  \]
(note that $ V^{\frac12} | V|^{\frac12} = V $).
Multiplying the identity $R_0(\lambda) = R_V(\lambda) +R_0(\lambda )VR_V(\lambda)$ from the left and the right by $V^{1/2}$ and $\abs{V}^{1/2}$ respectively, we then have (omitting $\lambda$ for improved readability)
\[   (\text{Id} + \mathbf  R_V ) V^{\frac12} R_V |V|^{\frac12} = \mathbf R_V  \ \Longrightarrow \
V^{\frac12} R_V |V|^{\frac12} =  (\text{Id} + \mathbf  R_V )^{-1} \mathbf R_V , \]
provided $\text{Id} + \mathbf  R_V$ is invertible. Inserting $ R_V = R_0 - R_V V R_0 $ into $ R_V = R_0 - R_0 V R_V $ we obtain
\[ R_V = R_0 - R_0 V R_0 + R_0 |V|^{\frac12}  V^{\frac12} R_V |V|^{\frac12} V^{\frac12} R_0 , \]
and hence, if $\text{Id} + \mathbf  R_V$ is invertible, we obtain the formal identity. 
\begin{equation}\label{eq:resol}
R_V = R_0 - R_0 V R_0 + 
R_0 |V|^{\frac12}   (\text{Id} + \mathbf  R_V )^{-1} \mathbf R_V  V^{\frac12} R_0 . 
\end{equation}

To establish \eqref{L2L2bound}, we need the following simple lemma, which will allow us to invert $\text{Id}+ \mathbf R_V$ on $S_{M,\delta}$ (and thus make use of \eqref{eq:resol}) using the Neumann series.
\begin{lemma}\label{inversionlemma}
There exist parameters $\delta,M>0$ such that for every $\lambda\in S_{M,\delta}$, there holds
\begin{equation}\label{Rv12bound}
\|\mathbf{R}_V\|_{L^2\to L^2}\leq \tfrac{1}{2}.
\end{equation}
\end{lemma}
\begin{proof}
From the representation formula for the free resolvent, it suffices to establish the estimate
\begin{equation}\label{youngineq}
|\lambda|^{-1}\sup_{x\in\mathbb{R}}\int_{\mathbb{R}}|V(x)|^{\frac{1}{2}}|V(y)|^{\frac{1}{2}}e^{|\Im\lambda||x-y|}\dd y\leq\tfrac{1}{2},\hspace{5mm}\lambda\in S_{M,\delta}.
\end{equation}
Since $V$ satisfies the sub-exponential bound $|V(x)|\leq Ce^{-\alpha|x|}$ for some $C,\alpha>0$, it follows that by taking $\delta$ small enough, there holds
\begin{equation*}
|\lambda|^{-1}\sup_{x\in\mathbb{R}}\int_{\mathbb{R}}|V(x)|^{\frac{1}{2}}|V(y)|^{\frac{1}{2}}e^{|\Im\lambda||x-y|}\dd y\leq \frac{C_{\alpha}}{M},\hspace{5mm}\lambda\in S_{M,\delta}.
\end{equation*}
The desired bound then follows by taking $M$ sufficiently large.
\end{proof}
By slightly modifying the above argument, we have for every $\rho\in C_c^{\infty}(\mathbb{R})$
\begin{equation}\label{freebounds}
\|\rho R_0 |V|^{\frac{1}{2}}\|_{L^2\to L^2}+\|V^{\frac{1}{2}}R_0\rho\|_{L^2\to L^2}\leq C_{\rho}
\end{equation}
for $\lambda\in S_{M,\delta}$ (where $\delta$ and $M$ are exactly as above, and do not depend on the choice of $\rho$). At this point, \eqref{L2L2bound} follows immediately from the formula \eqref{eq:resol}, \Cref{inversionlemma}, and \eqref{freebounds}. This completes the proof of \eqref{F:prop3}.
\medskip

We finally prove Property \ref{F:prop4}.
The bound is clear for $r\leq 1$ thanks to the fact that $F$ is non-constant and entire. We now assume $r\geq 1$. Thanks to \eqref{eq:boundonvp} and the fact that the Wronskian is constant in $x$, we have the upper bound
\begin{equation*}
|F(\lambda)|\lesssim e^{C|\lambda|\log(1+|\lambda|)}.
\end{equation*}
Moreover, since $F$ is non-constant, we can find $\lambda_0$ in the unit ball such that $F(\lambda_0)\neq 0$. Jensen's formula then yields
\begin{equation*}
 \#\set{z \in \C : F(z) =0 , |z| \le r}\leq C|r|\log(1+|r|)\lesssim_{\epsilon} r^{1+\epsilon},\hspace{5mm}r\geq 1
\end{equation*}
for some constant $C$ depending on $F(\lambda_0)$, but not on $r$.
\end{proof}

\appendix

\section{Resonant Condition Examples}\label{section:res_cond_examples}

Here we prove various examples coming from the resonant condition discussion in \Cref{ss:res_cond}.
We repeat the statements of the examples.

\begin{ex}[Regge--Wheeler potential in $\mathcal{RC}$]
If $\Lambda = 4 - \sqrt 5$ and $m = \frac{\sqrt 5 - 1 }{6}$, then for $\ell = 1$, $V_\ell \in \mathcal{RC}^+_2$ (recalling \eqref{eq:165}).
\end{ex}
\begin{proof}
We begin by explicitly computing the coefficients of the expansion $V_\ell$ near $+\infty$ in the $w = e^{-A_+ x}$ variable.

Let $\rho \coloneq r - r_+$ and $w \coloneq e^{-A_+ x}$. 
Near $r = r_+$, we can write
\begin{align}
    G(r) = - A_+ \rho +\tfrac 1 2 G'' (r_+) \rho^2 + O (\rho^3 )
\end{align}
so that 
\begin{align}
    \frac{1}{G(r)} = \frac{-1}{A_+ \rho} \left (  1 + \frac{G '' (r_+)}{2A_+} \rho + O(\rho^2 )\right)
\end{align}
therefore
\begin{align}
    x (\rho) = -\frac{1}{A_+} \log \rho  - \frac{G''(r_+)}{2A_+^2} \rho + O(\rho^2)+ c_0
\end{align}
where $c_0$ is a constant of integration.
We can then write
\begin{align}
    w &= e^{-A_+x} = \rho e^{G''(r_+)\rho A_+^{-1} / 2+ O (\rho^2)-A_+c_0}\\
    &= \rho(a_0 + a_1 \rho + O(\rho^2)).
\end{align}
Setting $\rho = b_1 w + b_2 w^2 + O(w^3)$, we see that $a_0 b_1 = 1$ and $a_0 b_2 + b_1^2 a_1 = 0$. 
Therefore
\begin{align}
    b_1 &= a_0^{-1} = e^{A_+ c_0}\\
    b_2 &= \frac{-b_1^2 a_1}{a_0}=-e^{3A_+ c_0} \left( e^{-A_+ c_0} \frac{G''(r_+)}{2A_+}\right) = -e^{2A_+ c_0} \frac{G''(r_+)}{2A_+}.
\end{align}
Recall that
\begin{align}
    G(r) r^{-2} (\ell (\ell+1) + rG'(r)) = V_\ell (r) = \sum_1^\infty \tilde V_j  \rho^j
\end{align}
we have
\begin{align}
    G(r) &= -A_+ \rho + \tfrac{1}{2} G'' (r_+) \rho^2 + O(\rho^3),\\
    r^{-2} &= r_+^{-2}  -2 r_+^{-3} \rho + O(\rho^2),\\
    \ell (\ell +1) + rG'(r) &= K_0 + K_1 \rho + O(\rho^2),
\end{align}
where $K_0 \coloneq \ell(\ell+1) - A_+ r_+$ and $K_1 \coloneq -A_+ +r_+G''(r_+)$. 
Therefore
\begin{align}
    \tilde V_1 &= -A_+ r_+^{-2} K_0 ,\\
    \tilde V_2 &= \tfrac{1}{2} G'' (r_+)r_+^{-2} K_0- A_+ (-2r_+^{-3}K_0 + r_+^{-2}K_1). 
\end{align}
But now
\begin{align}
    \sum_1^\infty \tilde V_j \rho^j = \tilde V_1 (b_1 w +b_2 w^2 + O(w^3)) + \tilde V_2 (b_1 ^2 w^2 + O(w^3)) + O(w^3)
\end{align}
so that $    V_1 = \tilde V_1 b_1,  \quad   V_2 = \tilde V_1 b_2 +\tilde V_2 b_1^2$.
We have $V _\ell \in \mathcal{RC}_2^+$ if and only if $V_1\neq 0$ and
\begin{align}
    V_2 =  \frac{V_1^2}{A_+^{2}}.\label{eq:505}
\end{align}
First, we check the non-degeneracy condition $V_1=0$ holds for every $\ell\in\mathbb{Z}_{\geq 0}$. For this we have to show that $\ell(\ell+1)-A_+r_+\neq 0$ whenever $\ell\in\mathbb{Z}_{\geq 0}$. From the fact that $l(l+1)$ is an even integer, this follows immediately from the following elementary lemma
\begin{lemma}\label{nondegenRW}
Under the constraints $m>0$ and $0<\Lambda<\frac{1}{9m^2}$, there holds
\begin{equation*}
0<A_+r_+<2.    
\end{equation*}
\end{lemma}
\begin{proof}
Since $G(r_+)=0$, we have
\begin{equation*}
\Lambda \frac{r_+^2}{3}=1-\frac{2m}{r_+}    
\end{equation*}
Substituting this into the formula
\begin{equation*}
G'(r)=\frac{2m}{r^2}-\frac{2\Lambda r}{3}    
\end{equation*}
gives
\begin{equation*}
A_+r_+=-G'(r_+)r_+=-\frac{2m}{r_+}+2\left(1-\frac{2m}{r_+}\right)=2-\frac{6m}{r_+}. 
\end{equation*}
Therefore, $A_+r_+ < 2$. 
It now suffices to show that $0< 2- 6m/r_+$.
Note that $G(3m)=\frac{1}{3}-3\Lambda m^2>0$, and since $G(r)\to -\infty$ as $r\to \infty$, we must have $r_+>3m$, which proves the lemma.
\end{proof}
It is worth remarking that the potential $V_l$ is also non-degenerate at the other root $r_-$. This just follows from the Taylor expansion $G'(r_-)=A_->0$ and the fact that $A_-r_->0$ (so, $l(l+1)+A_-r_-=0$ has no solutions for $l\in\mathbb{Z}_{\geq 0}$).
\\
 
To force \eqref{eq:505}, we set $\Lambda = -6m +3$ (with $\Lambda \in (0,1/(9m^2))$ so that $r_+=1$.
In this case, we get that $A_+ = 2(1-3m)$ and $G''(r_+) = -2$.
Therefore, letting $L= \ell (\ell +1)$,
\begin{align}
    \frac{V_1^2}{V_2} &= \frac{\tilde V_1 ^2 b_1^2}{\tilde V_1 b_2 + \tilde V_2 b_1^2}= \frac{\tilde V_1^2 }{A_+^{-1}\tilde V_1 +\tilde V_2 } = -\frac{A_+^2 K_0^2 }{2 K_0 +A_+ (-2K_0 +K_1)}\\
    &=-\frac{A_+^2 (L-A_+)^2}{2(L-A_+)+A_+(A_+ - 2 L - 2)}, \label{eq:512}
\end{align}
where we use that $b_2/ b_1^2 = A_+^{-1}$.
Setting \eqref{eq:512} equal to $A_+^2$, and solving for $A_+$, we get
\begin{align}
    A_+ = (1+L)\pm \tfrac 12 \sqrt{2(L^2 + 2L+2)}.\label{eq:517}
\end{align}
Restricting to $\ell = 1$, so $L=2$, and using that $A_+ = 2(1-3m)$, one of the values of \eqref{eq:517} provides a positive value of $m=(\sqrt 5-1)/6$.
It is then easy to verify in this case that $\Lambda = 4 -\sqrt5 \in (0 ,1/(9m^2))$.
\end{proof}

\begin{ex}[P\"oschl--Teller potential in $\mathcal{RC}$]\label{eq:posch_tel_ex}
The P\"oschl--Teller potential \eqref{eq:PTpotential} is not in $\mathcal{RC}$.
However for the  P\"oschl--Teller potential well, we have
\begin{align}
\frac{-m(m+1)}{\cosh^2 (x)} \in \mathcal{RC}_j^{\pm} \iff m\in \Z_{\ge 1} \text{ and }j\ge m+1.
\end{align}
\end{ex}

\begin{proof}
We will show $V\in \mathcal{RC}_j^+$ if and only if $m\in \Z\cap [1,j-1]$, the statement for $\mathcal{RC}_j^-$ is identical.

We can explicitly write the outgoing solutions using hypergeometric functions.
If $u(x) = e^{i\lambda x} v(w)$, with $w = e^{-2x}$, is outgoing at $+\infty$, then we can rewrite $-u'' + V(x) u = \lambda ^2 u$ as
\begin{align}
    w^2 v''(w) + w (1-\alpha)v'(w) + \frac{m(m+1) w}{(1+w)^2} v(w) = 0 \label{eq:1668}
\end{align}
where $\alpha = i\lambda$.
Next let $y = w/ (1+w)$, and set $f(y) = v(w)$, so that \eqref{eq:1668} becomes
\begin{align}
    y(1-y)f'' (y) + (c - (a+b+1)y) f'(y) - ab f(y) = 0
\end{align}
with $a= -m$, $b = m+1$, $c = 1-\alpha$.
This is the Gauss hypergeometric equation, whose solution normalized by $f(0)=1$ is
\begin{align}
    _2F_1(a,b;c;y) = \sum_{k=0}^\infty \frac{(a)_k (b)_k}{(c)_k k!} y^k
\end{align}
Where $(d)_k \coloneq \prod_{j=0}^{k-1} (d +j)$ is the rising Pochhammer symbol (with $(d)_0 \coloneq 1$).
Therefore
\begin{align}
    v(w) &= {}_2F_1\left ( -m,m+1;1-\alpha; \frac{w}{1+w}\right) \\
    &= (1+w)^{-m}{}_2F_1(-m,-m-\alpha ; 1-\alpha ;-w )\label{eq:1684}
\end{align}
where we used Euler's identity:
\begin{align}
    _2 F_1 (a,b;c;z) = (1-z)^{-a} {}_2F_1\left(a,c-b;c;\frac{z}{z-1}\right).
\end{align}
We can then rewrite \eqref{eq:1684} as
\begin{align}
    v(w) &= \left ( \sum_{n=0}^\infty \frac{(m)_n (-w)^n}{n!} \right)\left(\sum_{k=0}^\infty  \frac{(-m)_k (-m-\alpha)_k (-w)^k}{(1-\alpha)_k k!}\right)\\
    &=1 + \sum_{j=1}^\infty w^j  (-1)^j \left(\sum_{k=0}^j \frac{(-m-\alpha)_k(-m)_k(m)_{j-k}}{k!(j-k)!(1-\alpha)_k}\right) 
\end{align}
so that setting (for $j\ge 1$)
\begin{align}
    v_j(\alpha)\coloneq   (-1)^j \left(\sum_{k=0}^j \frac{(-m-\alpha)_k(-m)_k(m)_{j-k}}{k!(j-k)!(1-\alpha)_k}\right)v_0(\alpha) \label{eq:1694}
\end{align}
we get $v(w) = v(w,\alpha)  = \sum_{j=0}^\infty v_j(\alpha) w^j $.

Following Step 2 of the proof of \Cref{prop:existenceofoutgoing}, we write $T_j(\alpha) = 4 j (j-\alpha)$, and use that $T_j (\alpha ) v_j (\alpha) = S_j(\alpha)$ so that \eqref{eq:1694} gives
\begin{align}
    S_j(\alpha) =4j(j-\alpha) (-1)^j \sum_{k=0}^j \frac{(-m)_k (-m-\alpha)_k (m)_{j-k}}{(1-\alpha )_k k! (j-k)!} v_0(\alpha).
\end{align}
Following the proof of \Cref{prop:existenceofoutgoing}, $V\in \mathcal{RC}_j^\pm$ if and only if $\lim_{\alpha\to j} S_j(\alpha) v_0 (\alpha) ^{-1}=0$.
This limit can be computed by observing that 
\begin{align}
    \lim_{\alpha \to j}(j-\alpha) \sum_{k=0}^j \frac{(-m)_k (-m-\alpha)_k (m)_{j-k}}{(1-\alpha )_k k! (j-k)!}&=\lim_{\alpha \to j}(j-\alpha) \frac{(-m)_j (-m-\alpha)_j}{(1-\alpha )_j j! }\\
    &= \frac{(-1)^{j-1}(-m)_j (-m-j )_j}{(j-1)!j!}.
\end{align}
The first equality follows as for $k<j$, the factor $(1-\alpha)_k$ is nonzero at $\alpha=j$,
so these terms are holomorphic at $\alpha=j$ and vanish after
multiplication by $(j-\alpha)$. Thus only the $k=j$ term contributes.
This last expression is zero if and only if $(-m)_j = 0$ which is true if and only $m\in \Z \cap [0,j-1]$ (as $m=0$ gives a zero potential, we omit this value).

\end{proof}

\begin{corr}\label{ex:PT:resonances}
For the potential
\begin{align}
    V_m (x)  = \frac{-m(m+1)}{\cosh^2 (x)}
\end{align}
with $m \in \Z_{\ge 1}$, then $\Res(V_m) =  i (\Z \cap [-m,m])$ (where all poles have multiplicity $1$).
\end{corr}
\begin{proof}
By \cite{cevik2016resonances}, the transmission coefficient is 
\begin{align}
       T(\lambda) = \frac{\Gamma(m+1-i\lambda) \Gamma (-m-i\lambda)}{\Gamma(1-i\lambda)\Gamma(-i\lambda)}.
\end{align}
If $m \in \Z_{\ge 1}$, then this simplifies to
\begin{align}
    T(\lambda) = (-1)^m\prod_{k=1}^m \frac{k-i\lambda}{k+i\lambda}.
\end{align}
which has poles of order $1$ at $\lambda = ik$ for $k= 1,2,\dots, m$ and zeros of order $1$ at $\lambda = -ik$ for $k= 1,2,\dots, m$.

By \eqref{eq:598}
\begin{align}
    W[u_+,u_-](\lambda) = \frac{2i\lambda}{T(\lambda) \Gamma_{\Lambda_+} (1 -i\lambda)\Gamma_{\Lambda_-} (1-i\lambda)}.
\end{align}
By \Cref{eq:posch_tel_ex}, $\Lambda_\pm  = \set{j \in \Z : j \ge m+1}$, so that $\Gamma_{\Lambda_\pm} (1-i\lambda)$ has poles of order $1$ at $\lambda = -ik$ for $k = 1,2,\dots , m$.

Therefore the zeros of $W[u_+,u_-](\lambda)$ are exactly at $\lambda = ik$ for $k=1,2,\dots m$ (by the poles of $T(\lambda)$), $\lambda = 0$, and $\lambda = -ik$ for $k=1,2,\dots, m$ (by a double pole from $\Gamma_{\Lambda_\pm}$ canceling a single zero of $T(\lambda)$).
All these zeros are simple.

Then, by \Cref{lemma:inductthing}, the resonances of $V_m$ coincide (with multiplicity) with the zeros of $W[u_+,u_-](\lambda)$, which proves the corollary.
\end{proof}

\section{Proof of \texorpdfstring{\Cref{lemma:inductthing}}{Lemma 3.3}}\label{s:proof of lemmainduct}

In proving \Cref{Fprop}, we used \Cref{lemma:inductthing}, which states that the poles of $R_V(\lambda)$ coincide (with multiplicity) with the zeros of $F(\lambda) = W[u_+ (\cdot,\lambda), u_-(\cdot,\lambda) ]$.
We define the multiplicity of a resonance in the usual way.
\begin{defi}\label{def:multi}
For a potential $V(x)$ satisfying \eqref{hyp:1}, a resonance is a pole of the meromorphic continuation of the resolvent $R_V(\lambda)$ (recall \Cref{prop:meroresolv}).
If $R_V(\lambda)$ has a pole at $\lambda _0 \in \C$, then we can write
\begin{align}
    R_V(\lambda ) = \sum_{j=1}^J (\lambda - \lambda_0)^{-j}B_j +B_0 (\lambda)
\end{align}
for finite rank operators $B_j$, $B_0(\lambda)$ holomorphic near $\lambda = \lambda_0$, and $J \in \Z_{\ge 1}$ (the order of the pole).
The \textbf{multiplicity of the pole} $\lambda_0$ is the dimension of the space spanned by $B_j (L^2_c (\R))$ for $j=1,\dots ,J$.
\end{defi}

By \eqref{eq:535}, the poles of $R_V(\lambda)$ are exactly the poles of the resolvent function
\begin{align}
    R(x,y,\lambda) = \frac{u_+(x,\lambda) u _-(y,\lambda) H(x-y) + u_-(x,\lambda)u_+(y,\lambda) H(y-x) }{F(\lambda)}\label{eq:810}
\end{align}
where $H(\cdot)$ is the Heaviside function. We now prove \Cref{lemma:inductthing}. We begin with the first part, where we show that $\lambda_0$ is a pole of $R_V(\lambda)$ with multiplicity $m$ if and only if $\lambda_0$ is a zero of order $m$ for $F(\lambda)$. Clearly, from the formula for $R(x,y,\lambda)$, if $F(\lambda_0)\neq 0$, then $\lambda_0$ is not a pole of the resolvent. Hence, it suffices to show that if $\lambda_0\in\mathbb{C}$ is a zero of order $m\geq 1$ for $F$, if and only if $\lambda_0$ is a pole of $R_V(\lambda)$ with multiplicity $m$. We will begin with the forward implication, which is more difficult. In our proof, we will need the following technical lemma relating the spaces spanned by the $\lambda$ derivatives of $u_+$ and $u_-$ at $\lambda=\lambda_0$. 
\begin{lemma}\label{ordwronsk}
If $F(\lambda)$ vanishes to order $m\geq 1$ at $\lambda_0\in\mathbb{C}$, then there exists a polynomial $p_{m-1}(\lambda)$ of degree $m-1$ (depending only on $\lambda)$ such that the function
\begin{equation}\label{rdef}
r_{m-1}(x,\lambda)\coloneq u_+(x,\lambda)-p_{m-1}(\lambda)u_-(x,\lambda)     
\end{equation}
vanishes to order $m$ at $\lambda=\lambda_0$. That is, $r_{m-1}$ satisfies
\begin{equation}\label{vanishingcond}
(\partial_{l}^{\lambda}r_{m-1})_{|\lambda=\lambda_0}=0,\hspace{5mm}0\leq l\leq m-1.    
\end{equation}
In particular, we have
\begin{equation*}
\Span\{(\partial_{\lambda}^ju_+)_{|\lambda=\lambda_0} \}_{j=0}^{m-1}\subset\Span\{(\partial_{\lambda}^ju_-)_{|\lambda=\lambda_0}\}_{j=0}^{m-1}.
\end{equation*}
\end{lemma}
\begin{proof}
 First, the case $m=1$ is obvious as the $W(u_+,u_-)(\lambda_0)=0$ implies that $u_+$ and $u_-$ are linearly dependent at $\lambda_0$.
 For $m >1$, we construct $r_{m-1}$ by induction (for fixed $m$).
 Assume that for $0\leq j<m-1$, we have constructed $r_0,...,r_j$ satisfying \eqref{rdef} and \eqref{vanishingcond} (the case $j=0$ follows easily). 
 We now aim to construct $r_{j+1}$. 
 First, we observe that there holds
\begin{equation}\label{wronskid}
\partial_{\lambda}^{j+1}W(u_+,u_-)_{|\lambda=\lambda_0}=\partial_{\lambda}^{j+1}W(r_{j},u_-)_{|\lambda=\lambda_0}=W(\partial^{j+1}_{\lambda}r_j,u_-)_{|\lambda=\lambda_0}=0  
\end{equation}
where in the first equality, we used the formula \eqref{rdef} and in the second, we used the Leibniz rule and \eqref{vanishingcond}. 
Moreover, we claim that $(\partial_{\lambda}^{j+1}r_j)(x,\lambda_0)$ satisfies $(P_V-\lambda_0^2 ) (\p_\lambda ^{j+j}r(x,\lambda_0))$.
To see this, we compute
\begin{equation*}
0=\partial_{\lambda}^{j+1}((P_V-\lambda^2)r_j)=(P_V-\lambda^2)\partial_{\lambda}^{j+1}r_j-2(j+1)\lambda\partial_{\lambda}^{j}r_j-j(j+1)\partial_{\lambda}^{j-1}r_j  
\end{equation*}
and see by \eqref{vanishingcond} that the second two terms on the right-hand side vanish. This, with \eqref{wronskid}, shows that there exists a constant $c_j$ such that
\begin{equation*}
(\partial_{\lambda}^{j+1}r_j)(x,\lambda_0)=c_ju_-(x,\lambda_0).    
\end{equation*}
We then define $r_{j+1}$ by correcting $r_j$ by a suitable degree $j+1$ monomial. We can simply take
\begin{equation*}
r_{j+1}(x,\lambda)\coloneq r_j(x,\lambda)-\frac{c_j}{(j+1)!}(\lambda-\lambda_0)^{j+1}u_-(x,\lambda).    
\end{equation*}
or equivalently
\begin{equation*}
p_{j+1}(\lambda)\coloneq p_j(\lambda)+\frac{c_j}{(j+1)!}(\lambda-\lambda_0)^{j+1}.    
\end{equation*}
It is straightforward to verify that $r_{j+1}$ satisfies \eqref{rdef} and $\eqref{vanishingcond}$. This completes the proof.
\end{proof}

Now, returning to the proof of \Cref{lemma:inductthing}, let us suppose $F$ has a zero of order $m\geq 1$ at $\lambda_0 \in \C$. 
Then $F = (\lambda - \lambda_0)^m G(\lambda)$ for $G$ an entire function nonzero near $\lambda_0$.
The Schwartz kernel of the resolvent can be written
\begin{align}
    R(x,y,\lambda) = \frac{u_+ (x,\lambda)u_- (y,\lambda ) H(x-y)  + u _- (x,\lambda) u_+(y,\lambda) H(y-x)}{(\lambda - \lambda_0 )^m G(\lambda)}.\label{eq:1117}
\end{align}
By applying \Cref{ordwronsk} and multiplying \eqref{rdef} by $u_{-}$, we find that there exists a polynomial $p(\lambda)$ 
\begin{equation*}
     \p_\lambda^j \big|_{\lambda_0} (u_+(y,\lambda)u_-(x,\lambda))=\p_\lambda^j \big|_{\lambda_0} (u_+(x,\lambda)u_-(y,\lambda))= \p_\lambda^j\big|_{\lambda_0}(p(\lambda)u_-(x,\lambda)u_-(y,\lambda)),
\end{equation*}
for $0\leq j\leq m-1$. Therefore, by Taylor expanding $u_+(x,\lambda)u_{-}(y,\lambda)$ and $u_-(x,\lambda)u_+(y,\lambda)$ about $\lambda=\lambda_0$ and using the identity $H(x-y)+H(y-x) = 1$, we can write the Schwartz kernel of the resolvent in \eqref{eq:1117} as
\begin{align}
    R(x,y,\lambda) = \sum_{j=0}^{m-1} \frac{\p_\lambda^j \big|_{\lambda_0}\big(p(\lambda)u_-(x,\lambda)u_-(y,\lambda))}{j!G(\lambda)(\lambda -\lambda_0)^{m-j}} + \tilde{R}(x,y,\lambda)
\end{align}
with $\tilde{R}(x,y,\lambda)$ holomorphic in $\lambda$ near $\lambda_0$.
For notational convenience, we define $T_j(\lambda)$ as the operators with Schwartz kernel
\begin{align}
    T_j(x,y,\lambda) \coloneq \p_\lambda^j(p(\lambda)u_-(x,\lambda)u_-(y,\lambda))
\end{align}
so that 
\begin{align}
    R(x,y,\lambda) = \sum_{j=0}^{m-1} \frac{T_j(x,y,\lambda_0)}{j!G(\lambda)(\lambda - \lambda _0)^{m-j}} + \tilde{R}(x,y,\lambda).
\end{align}
Let $B_j$ be the operator with kernel $T_j(x,y,\lambda_0)$.
It suffices to show that
\begin{align}
    \dim (\Span(B_0(L^2_{c}),B_1(L^2_{c}),\dots,B_{m-1}(L^2_{c})) = m. \label{eq:1188}
\end{align}
To this end, we first observe that
\begin{align}
    T_0 (x,y,\lambda) = p(\lambda) u_-(x,\lambda) u_-(y,\lambda)\label{eq:1192}
\end{align}
so that $B_0$ is a rank one operator and $\Span (B_0 (L^2_c)) = \Span(u_-(x,\lambda_0))$.
For $j\geq 1$, by the Leibniz rule, $T_j(x,y,\lambda)$ is of the form
\begin{equation}\label{eq:1209}
    T_{j}(x,y,\lambda) =  p(\lambda)\p_\lambda^{j} u_-(x,\lambda)u_-(y,\lambda) +  \sum_{\substack{|\alpha |<j}} c_{\alpha} \p_\lambda^{\alpha_1}p(\lambda) \partial_\lambda^{\alpha_2} u_-(x,\lambda) \partial _\lambda^{\alpha_3} u_-(y,\lambda) 
\end{equation}
for some coefficients $c_\alpha\in \Z_{\ge 0}$. To determine the rank of $B_j$, we need the following lemma.
\begin{lemma}\label{lemma:1222}
For each $j = 1,\dots, m$,
    $\p_\lambda ^j\big|_{\lambda_0} u_-(x,\lambda)$ is not in the span of $$\set{\p_\lambda^0\big|_{\lambda_0} u_-(x,\lambda),\p_\lambda \big|_{\lambda_0} u_-(x,\lambda),\dots,\p_\lambda ^{j-1}\big|_{\lambda_0} u_-(x,\lambda)}.$$
\end{lemma}
\begin{proof}
Our first observation is to notice that by differentiating the equation $(P-\lambda^2)u_-=0$ a total of $1\leq \ell \le j$ times in $\lambda$, there holds 
\begin{equation}\label{eq:1172}
(P-\lambda_0^2)(\p_\lambda ^l\big|_{\lambda_0}u_-(x,\lambda))=2l\lambda_0\p_\lambda ^{l-1}\big|_{\lambda_0}u_-+(l-1)l\p_\lambda ^{l-2}\big|_{\lambda_0}u_-
\end{equation}
with the convention that $(l-1)l\p_\lambda ^{l-2}\big|_{\lambda_0}u_-=0$ when $l=1$.
Note that the right-hand side of \eqref{eq:1172} involves derivatives of order at most $\ell -1$.
We can therefore inductively apply \eqref{eq:1172} to get that for each $\ell \in \Z_{\ge 0}$
\begin{align}
    (P-\lambda_0^2)^{\ell + 1}(\partial_\lambda^\ell\big|_{\lambda_0} u_-(x,\lambda)) = 0. \label{eq:1177}
\end{align}

Let $j>1$ and suppose by contradiction that $\p_\lambda ^j\big|_{\lambda_0} u_-(x,\lambda)$ is in the span of the vectors $\set{\p_\lambda^\ell|_{\lambda_0} u_-(x,\lambda)}_{\ell=0}^{j-1}$.
By \eqref{eq:1177}, $(P-\lambda_0^2)^{j} (\p_\lambda ^j\big|_{\lambda_0}u_-(x,\lambda)) = 0$.
While on the other hand, by repeatedly using \eqref{eq:1172},
\begin{align}
    0= (P-\lambda_0^2)^{j} \left(\p_\lambda ^j\big|_{\lambda_0}u_-(x,\lambda)\right) =  c u_-(x,\lambda_0)
\end{align}
for some nonzero $c \in \C$. 
Because $u_-(x,\lambda_0)$ is not identically zero, this gives a contradiction.
\end{proof}
We will also need the following lemma, which along with \Cref{lemma:1222} proves \eqref{eq:1188}
\begin{lemma}\label{lemma:1240}
    For $j=0,\dots, m$:
    \begin{equation}\label{eq:1241}
\begin{split}
   &\Span(B_0(L^2_{c}),B_1(L^2_{c}),\dots,B_{j-1}(L^2_{c})) \\
    &\qquad \qquad =  \Span(\p^0_\lambda\big|_{\lambda _0}u_-(x,\lambda), \p^1_\lambda \big|_{\lambda _0}u_-(x,\lambda), \dots, \p^{j-1}_\lambda \big|_{\lambda _0}u_-(x,\lambda)).
    \end{split}
\end{equation}
\end{lemma}
\begin{proof}
The case where $j=0$ was established by \eqref{eq:1192}.
We now assume \eqref{eq:1241} is true for a $j\leq m-1$.
Let $u \in L^2_c$ be such that $\int u(y) u_-(y,\lambda)\dd y\neq 0$, so by \eqref{eq:1209}
\begin{align}
B_{j} u = \sum_{\ell=0}^{j} c_\ell \p_\lambda^\ell|_{\lambda_0}u_-(x,\lambda) 
\end{align}
for $c_\ell \in \C$ and $c_{j} \neq 0$.
Therefore, by the inductive hypothesis, we have $\p_\lambda^{j}|_{\lambda_0}u_-(x,\lambda)\in\Span(B_0(L^2_{c}),B_1(L^2_{c}),\dots,B_{j}(L^2_{c}))$, and thus
\begin{align}
    &\Span(\p^0_\lambda |_{\lambda _0}u_-(x,\lambda), \p^1_\lambda |_{\lambda _0}u_-(x,\lambda), \dots, \p^{j}_\lambda |_{\lambda _0}u_-(x,\lambda)) \\
    &\qquad\qquad\qquad\qquad\qquad\subset\Span(B_0(L^2_{c}),B_1(L^2_{c}),\dots,B_{j}(L^2_{c})).
\end{align}
But we also immediately have the converse inclusion as for $u\in L^2_c$, $B_{j} u$ is a linear combination of $\p^\ell_\lambda|_{\lambda_0}u_-(x,\lambda)$ for $\ell=0,\dots, j$.
\end{proof}
We have now shown that if $F(\lambda)$ is a zero of order $m$ at $\lambda_0$, then the resolvent $R(\lambda)$ has a pole of multiplicity $m$.

For the converse direction, we require the following Lemma.
\begin{lemma}\label{lemma:1281}
    If the resolvent $R(\lambda)$ has a pole of order $m \in \Z_{>0}$ at $\lambda_0$, then $F(\lambda)$ has a zero of order $m$ at $\lambda_0$.
\end{lemma}
\begin{proof}
Given the hypothesis, we may write the Schwartz kernel of $R(\lambda)$ as
\begin{align}
    R(\lambda ) = \sum_{j=0}^m\frac{R_j(x,y,\lambda)}{(\lambda - \lambda_0)^{m-j}}
\end{align}
with $R_j$ holomorphic near $\lambda_0$ and $R_0 \neq 0$.
This must agree with the expression of $R(\lambda)$ given in \eqref{eq:810}, which immediately implies that $F$ has a zero of order $m$ at $\lambda_0$.
\end{proof}

We now suppose that $R(\lambda)$ has a resonance of multiplicity $m$ at $\lambda_0$.
By \Cref{lemma:1281}, this implies the order of the resonance is $m$.
Indeed, if the order of the resonance at $\lambda_0$ was $m' \neq m$, then by \Cref{lemma:1281}, $F$ would have a zero of order $m'$ at $\lambda_0$. 
We could then apply the forward direction to see that this implies the multiplicity of the resonance is $m'$, which is a contradiction.
\\

To conclude the proof of \Cref{lemma:inductthing}, it remains to show that if $0$ is a resonance, then it is simple. For this, we recall from \eqref{eq:598} that the transmission coefficient $T(\lambda)$ can be expressed in terms of $F(\lambda)$ as  {
\begin{align}
    T(\lambda ) = \frac{2i\lambda}{\Gamma_{\Lambda_+}(1-\alpha_+)\Gamma_{\Lambda_-}(1-\alpha_-)F(\lambda)}.\label{eq:1184}
\end{align}}
By \cite[Theorem 2.1]{MKlaus_1988}, for a large class of potentials (which include $V$), $T(0) \neq 0$.
If $0$ is a resonance of $-\p_x^2 + V(x)$, then by \eqref{eq:1184} and that $T(0)\neq 0$, $F(\lambda)$ must have a zero of order $1$.
So by the above argument, $0$ is a resonance of multiplicity $1$. This finally concludes the proof of \Cref{lemma:inductthing}.

\section{Birman-Kre\u{\i}n trace formula for exponentially decaying potentials}\label{appendix}

In this section we prove the Birman-Kre\u{\i}n trace formula for exponentially decaying potentials (i.e., where $\abs{V(x)} \le e^{-c|x|}$).
The assumption on the decay is stronger than necessary but allows an easy definition 
of the multiplicity of the resonance at zero as then the resolvent continues meromorphically to a strip.

\begin{theorem}[\bf Birman--Kre\u{\i}n trace formula in one dimension] 
\label{t:BK}
Let $V \in C^\infty (\R ; \R)$ be such that $\abs{V(x)}\le e^{-c|x|}$ for some $c>0$, and set $P_V = D_x^2 + V(x)$.
Then for  $ f \in  \mathscr S ( \mathbb{R} ) $ 
the operator $    f ( P_V ) - f ( P_0) $ is of trace class and
\begin{equation}
\label{eq:BK1} 
 \begin{split}   \tr \left( f ( P_V ) - f ( P_0 ) \right)  = &
\, \frac{1}{ 2
  \pi i } \int_0^\infty  f ( \lambda^2 ) \tr \left( S ( \lambda)^{-1} \partial_\lambda S (
\lambda )\right) \dd \lambda \\ &  
\ \ \ \ \ \ \ \ + \sum_{j=1}^k f ( E_j  ) 
+ \tfrac12 ( m_R ( 0 ) - 1 ) f ( 0 ) 
\,, 
\end{split}
\end{equation}
where $ S ( \lambda ) $ is the scattering matrix, $m_R(0)$ is the multiplicity of the resonance at $0$, and 
$E_j$ are the (negative) eigenvalues of 
$ P_V $.
\end{theorem}
The proof of \Cref{t:BK} is a mild modification of the proof of the one-dimensional Birman-Kre\u{\i}n trace theorem for compactly supported potentials found in \cite[\S 2.6]{dyatlov2019mathematical}.

Our proof critically uses a determinant identity to analyze the scattering matrix.
A version of this identity was proven in \cite{froese1997asymptotic}, but for super-exponentially decaying potentials.
A minor adjustment to the proof will enable us to prove the identity for $\{\operatorname{Im}(\lambda)>-\delta\}\setminus\{0\}$. 
This is a weaker result, but will suffice for our purposes. 
The main technical input is the following mild adaptation of \cite[Lemma 7.6]{froese1997asymptotic} to our setting. 
\begin{lemma}\label{Froeseadaptation}
Let $V\in L^{\infty}(\mathbb{R})$ be a function satisfying the exponential decay bound
\begin{equation*}
|V(x)|\leq C e^{-2c_0|x|}    
\end{equation*}
for constants $C,c_0>0$.
Then for $\lambda\in\mathbb{C}$ such that $0<4|\operatorname{Im}(\lambda)|<c_0$, $\mathbf{R}_V (\lambda)$ is trace class and 
\begin{align}\label{eq:R2S}
     \det S ( \lambda ) = \frac{ \det (\rm{Id} +   \mathbf R_V ( - \lambda ) )}{ \det ( \rm{Id}  + \mathbf R_V ( \lambda ) ) },
\end{align}
where $\mathbf R_V ( \lambda ) \coloneq V^{\frac12} R_0 ( \lambda)|V|^{1/2}$ (which was introduced in the proof of \Cref{lemma:inductthing}). 
\end{lemma}
The proof follows by making minor adjustments to the argument in \cite{froese1997asymptotic} (specifically \cite[Lemma 3.1 and 7.6]{froese1997asymptotic}). We include the short argument for the convenience to the reader. 
\begin{proof}
\Step The first step is to show that 
\begin{equation}\label{limittrace}
\lim_{L\to\infty}V^{\frac{1}{2}}\mathbf{R}_{\chi_L}(\lambda)|V|^{\frac{1}{2}}=\mathbf{R}_V(\lambda)   
\end{equation}
where the limit is taken in the trace norm and $\chi_L$ denotes the indicator function on the interval $[-L,L]$.

The hypothesis on $V$ ensures that we have the bound $|V|^{\frac{1}{2}}\leq C e^{-c_0|x|}$. Our first goal will be to show that it suffices to establish \eqref{limittrace} with $V$ replaced by a simpler step function approximation $w$ of $C e^{-c_0|x|}$. To define $w$, we define the dyadic-like telescoping sequence $\mu_k$ by
\begin{equation*}
\mu_1\coloneq Ce^{-c_0},\hspace{5mm} \mu_k\coloneq C(e^{-c_0 k}-e^{-c_0 (k+1)})
\end{equation*}
where $k>1$. We then define $w$ by
\begin{equation*}
w(x)=\sum_{k=1}^{\infty}\mu_k\chi_k(x)    
\end{equation*}
where $\chi_k$ is the indicator function on $[-k,k]$. As a consequence of the telescoping identity
\begin{equation*}
\sum_{k\geq j}\mu_k=Ce^{-c_0 j},    
\end{equation*}
there holds $|V|^{\frac{1}{2}}w^{-1}\leq 1.$ Consequently, it suffices to establish \eqref{limittrace} with $V^{\frac{1}{2}}$ and $|V|^{\frac{1}{2}}$ replaced by $w$. To this end, we estimate
\begin{equation}\label{traceestseq}
\begin{split}
\|w\chi_LR_0\chi_Lw&-wR_0w\|_{1}
\\&=\|w\chi_LR_0(\lambda)\chi_Lw-wR_0(\lambda)\chi_Lw+wR_0(\lambda)\chi_Lw-wR_0(\lambda)w\|_1   
\\
&\leq 2\|wR_0(\lambda)(1-\chi_L)w\|_1
\\
&\leq 2\|\sum_{i\geq 1}\mu_i\chi_iR_0(\lambda)\sum_{j>L}\mu_j\chi_j\|_1
\\
&\leq 2\sum_i\sum_{j>L}\mu_i\mu_j\|\chi_iR_0(\lambda)\chi_j\|_1
\\
&\leq 2\sum_{i\leq j}\sum_{j>L}\mu_i\mu_j\|\chi_jR_0(\lambda)\chi_j\|_1+2\sum_{i>j}\sum_{j>L}\mu_i\mu_j\|\chi_iR_0(\lambda)\chi_i\|_1
\\
&\leq \frac{C}{|\operatorname{Im}(\lambda)|}\left(\sum_{i\geq 1}\mu_i\sum_{j>L}je^{j(4|\operatorname{Im}(\lambda)|-c_0)}+\sum_{i\geq 1}ie^{i(4|\operatorname{Im}(\lambda)|-c_0)}\sum_{j>L}\mu_j\right)
\end{split}    
\end{equation}
where in the last line, we used the elementary bound (see for instance, \cite [Corollary 7.5]{froese1997asymptotic}) 
\begin{equation*}
\|\chi_kR_0(\lambda)\chi_k\|_1\leq \frac{C}{|\operatorname{Im}(\lambda)|}ke^{4k|\operatorname{Im}(\lambda)|}.    
\end{equation*}
Thanks to the hypothesis $|\operatorname{Im}(\lambda)|<\frac{c_0}{4}$ and the summability of $\mu_i$, it follows that the last line of \eqref{traceestseq} goes to zero as $L\to\infty$.

\Step \cite[Lemma 3.1]{froese1997asymptotic} follows without modification. 
Indeed, by \cite[Proposition 7.2]{froese1997asymptotic},
\begin{align}
    \textbf{R}_{\chi_1}(\lambda)=\textbf{R}_{\chi_1} (-\lambda)+F_2(\lambda)\label{eq:1599}
\end{align}
where $F_2(\lambda)$ is a rank two operator.
Conjugating \eqref{eq:1599} by unitary dilation $\psi (x) \mapsto L^{-1/2}\psi(x/L)$, multiplying the left by $V^{1/2}$ and right by $|V|^{1/2}$, and using \eqref{limittrace}, we get
\begin{align}
    \textbf{R}_V(\lambda) = \textbf{R}_V(-\lambda)+F_V (\lambda)\label{eq:1604}
\end{align}
where $F_V(\lambda)$ is a rank two operator.
Rearranging \eqref{eq:1604}, and taking the Fredholm determinant gives \eqref{eq:R2S} (see \cite[eq. 3.4]{froese1997asymptotic}, and subsequent discussion).

\setcounter{step}{0}
\end{proof}

\begin{proof}[Proof of Theorem \ref{t:BK}]
For simplicity, we assume that there are {\em no} negative eigenvalues as their contribution is easy to analyze.

\Step We first show that $f(P_V) - f(P_0)$ is trace class.
Because $ P_V $ is self-adjoint, we have by Stone's formula
\begin{equation*}
\begin{split} 
f ( P_V ) & = \frac{ 1 } {2 \pi i } \int_0^\infty f ( \lambda^2 ) ( R_V (
\lambda ) - R_V ( - \lambda ) ) 2 \lambda \dd \lambda \\
& = \frac{1}{ 4 \pi i } \int_\mathbb{R} f ( \lambda^2 ) ( R_V (
\lambda ) - R_V ( - \lambda ) ) 2 \lambda \dd \lambda,
\end{split}
\end{equation*}
using that the integrand is even in $ \lambda $. 
The integral on the right-hand side should be interpreted as an operator $ L^2_{\rm{c} } \to L^2_{\rm{loc}} $. 
We therefore have
\begin{align}
    f(P_V)-f(P_0)&=\frac{1}{4\pi i } \int_\R f(\lambda ^2 ) (R_V(\lambda)-R_0 (\lambda ) -(R_V (-\lambda )-R_0 (\lambda ))2\lambda \dd \lambda\\
    &= \frac{1}{4\pi i } \int_\R f(\lambda^2 ) (R_V(-\lambda)VR_0(-\lambda)-R_V(\lambda )VR_0(\lambda ))2\lambda \dd \lambda,\label{eq:1413}
\end{align}
where we use that
\begin{align}
    R_V(\lambda )-R_0(\lambda ) = -R_V (\lambda )VR_0(\lambda).\label{eq:1474}
\end{align}

We now define the meromorphic family of operators
\begin{align}
    B(\lambda ) \coloneq 2\lambda R_V (\lambda ) VR_0(\lambda ) \colon L^2_{\rm{c} } \to L^2_{\rm{loc}}.\label{eq:Bla}
\end{align}
We observe that this family is holomorphic in the closed upper half plane  $\{\Im{\lambda} \ge 0\}$.
This is because $R_V (\lambda )VR_0(\lambda)$ has a simple pole at $\lambda =0$ (canceling the $\lambda$ is the definition of $B(\lambda)$).
This follows by using \eqref{eq:1474} and \Cref{lemma:inductthing}.

We can then rewrite \eqref{eq:1413} as
\begin{align}
    f(P_V ) - f(P_0) = \frac{1}{4\pi i }\int_\R f(\lambda^2 ) (B(-\lambda) -B(\lambda) )\dd \lambda.\label{eq:fpV} 
\end{align}
Fixing $\e>0$ sufficiently small, let $g \in \mathcal S(\C)$, $\supp g\subset \set{\lambda \in \C : \Im{\lambda }<\e}$, be an \textit{almost analytic extension} of $f(\lambda)$ (see for instance \cite[\S B.2]{dyatlov2019mathematical}):
\begin{align}
    g(\lambda ) = f(\lambda^2 ) ,\ \lambda \in \R ,\ \dbar _\lambda g(\lambda ) = \mathcal{O}(\abs{\Im \lambda}^\infty).\label{eq:faag}
\end{align}
The support of $g$ can be made arbitrarily close to the real axis.
This is critical, as we will eventually use \Cref{Froeseadaptation} -- which is only valid in a neighborhood of the real axis (which depends on the decay of $V$ i.e., $A_\pm$).
The Cauchy--Green formula \cite[(D.1.1)]{dyatlov2019mathematical}  applied to the right-hand
side of \eqref{eq:fpV} shows that 
\begin{align}\label{eq:1489}
f(P_V ) - f(P_0) = \sum_{\pm}\pm\frac{1}{2\pi}\int_{\Im \lambda > 0} \dbar _\lambda g(\lambda) B(\pm\lambda) \dd m(\lambda).
\end{align}

Since there are no negative eigenvalues, the spectral theorem gives the bound 
\begin{equation}
\label{eq:RVbound} \| R_V ( \lambda ) \|_{ L^2 \to L^2 } = \frac1 { d
  ( \lambda^2 , \mathbb{R}_+) } \leq  \frac{1} { | \lambda | \Im \lambda }
\,, \ \ \Im \lambda > 0 \,. \end{equation}
In particular, when $ V = 0 $ we have the following $L^2\to H^2$ bound when $ \Im \lambda > 0 $,
\[ \begin{split} \| \lambda R_0 ( \lambda ) \|_{ L^2 \to H^2 } & \lesssim | \lambda | ( \| D_x^2 R_0
( \lambda ) \|_{ L^2 \to L^2 } + \| R_0 ( \lambda ) \|_{ L^2 \to L^2 }
) \\
 & \lesssim | \lambda | ( 1 +  |\lambda |^2 ) \| R_0 ( \lambda ) \|_{
  L^2 \to L^2 } \\
& \lesssim   \frac{( 1 + | \lambda|^2 )}{\Im \lambda} . 
\end{split} \]
Combining this estimate with the sub-exponential bound $ |V | \leq e^{ - c|x| } $, ensures
$ V R_0 ( \lambda ) $ is a trace class operator for $ \Im \lambda > 0 $ and that
\[  \| \lambda V  R_0 ( \lambda) \|_{ \mathcal L_1 } \lesssim \frac{ 1 + |\lambda|^2 } { \Im \lambda } , \ \ \Im \lambda > 0 . \]
 From this bound, as well as \eqref{eq:Bla} and \eqref{eq:RVbound}, we find that for $\Im\lambda>0$, there holds
\begin{equation}
\label{eq:Blah} \|  B( \lambda ) \|_{ {\mathcal L}_1} \lesssim 
\frac{ 1 + |\lambda|^2   } { |\Im
  \lambda |^2  |\lambda |}   \lesssim \frac{  1 + |\lambda|^2  } { |\Im
  \lambda |^3}. \end{equation}

Using \eqref{eq:Blah} and \eqref{eq:faag} we conclude that for any $ N > 0 $, and in particular for $ N \geq 4 $,
\begin{align}
     &\norm{ \int_{\pm\Im\lambda > 0} \dbar _\lambda g(\lambda) B(\pm \lambda ) \dd m(\lambda) }_{{\mathcal L}_1} \\
     &\qquad\leq C_N 
\int_{ 0 < \pm \Im \lambda < 1 }   |\Im \lambda|^N
 ( 1 + |\lambda |)^{-N + 2 }  | \Im \lambda |^{-3 } \dd { m}
 (  \lambda ) < \infty \,. 
\end{align}
It follows that $f ( P_V ) - f ( P_0 )$ is trace class, as desired.
\medskip

\Step 
We now relate $B(\lambda)$ to the scattering matrix.

Define the operator
\begin{align}
   \mathcal{R}(\lambda) \coloneq  -\p_\lambda \textbf{R}_V(\lambda) (\text{Id} + \textbf{R}_V(\lambda))^{-1}.
\end{align}

This operator is related to the scattering matrix.
Indeed, by taking the logarithmic derivative of both sides of \eqref{eq:R2S}, we get
\begin{equation}
\label{eq:Fsc}  \tr ( \partial_\lambda S( \lambda ) S ( \lambda )^{-1}) = \tr \mathcal R (\lambda) + \tr \mathcal R (-\lambda).
\end{equation}

\begin{lemma}\label{lem:1560}
    For $\Im \lambda > 0$, the operator $\mathcal{R}(\lambda)$ is trace class and $\tr \mathcal R (\lambda) = \tr B(\lambda)$.
\end{lemma}
\begin{proof}

{
We observe that for $\lambda \in \C$, $ \mathcal R ( \lambda ) $ is a meromorphic
family of operators in $ \mathcal L_1 ( L^2  ) $, and has no poles in the region $\Im \lambda > 0 $ (recalling our assumption that there are no negative eigenvalues).
We then have a direct analog of \cite[(2.2.33)]{dyatlov2019mathematical}:
\[ \text{Id} + \mathbf  R_V ( \lambda ) = U_1 ( \lambda ) ( Q_0 + \lambda^{-1} Q_{-1} + \lambda Q_1 ) U_2 ( \lambda ) , \]
where $ U_j ( \lambda ) $ are invertible, holomorphic, and $Q_j$ are finite rank operators such that
\[  \text{rank}\hspace{1mm} Q_{-1} = 1 , \ \ \text{rank}\hspace{1mm} Q_1 = m_R ( 0 ) , \ \ Q_j Q_k = \delta_{jk} Q_j , \]
see \cite[Theorem C.10]{dyatlov2019mathematical}. }
Hence, 
\begin{equation}
\label{eq:trF0} \tr \mathcal R ( \lambda ) = - \tr ( \lambda^{-1} Q_1 - \lambda^{-1} Q_{-1 } ) +
\varphi ( \lambda ) = \frac{ 1 -  m_R ( 0 ) } \lambda  +  \varphi (
\lambda ) , 
\end{equation}
where $ \varphi ( \lambda ) $ is holomorphic in $ \Im \lambda \geq 0 $. 
Equation \eqref{eq:trF0} follows by writing
\begin{align}
    \tr \mathcal{R}(\lambda ) = &-\p_\lambda \log \det (U_1 (\lambda ) ) - \p _\lambda \log \det (U_2 (\lambda )) \\
    &-\log (\det (Q_0 + \lambda ^{-1} Q_{-1} + \lambda Q_1)),
\end{align}
using that $Q_j Q_k = \delta_{jk}Q_j$ (and Jacobi's formula) to rewrite the third term as $$-\tr (\lambda ^{-1} Q_1 - \lambda ^{-1}Q_{-1}). $$

{
In order to control $ \varphi ( \lambda ) $, we first observe the $L^2\to H^2$ bound for $\mathbf{R}_V$
\[  \| \mathbf R_V  ( \lambda )  \|_{ L^2 \to H^2  }
\leq C |\lambda| e^{ C ( \Im \lambda )_- } , \ \ |\lambda | 
\geq 1 . \]
From the Cauchy formula, we also have that for $ \Im \lambda \geq 0 $, $ |\lambda | \gtrsim 1 $, 
\[   \| \partial_\lambda \mathbf R_V ( \lambda ) \|_{ \mathcal L_1 } 
\leq  C \| \partial_\lambda \mathbf   R_V ( \lambda )  \|_{ L^2 \to H^2 }
\leq C' | \lambda | . \]
Combining this bound with the definition of $ \mathcal R ( \lambda ) $, and the 
invertibility of  $ I + \mathbf  R_V ( \lambda )  $ in the region 
$ |\lambda | \gg 1 $, $ \Im \lambda \geq 0 $, we obtain for some $R_0>0$
\[  |\tr \mathcal R ( \lambda ) | \leq C | \lambda | , \ \ 
| \lambda | \geq R_0 , \ \ \Im \lambda \geq 0 . \]
Then, since $ \varphi ( \lambda )  $ is holomorphic in the region $ \Im \lambda \geq 0 $, we obtain the bound
\[  | \varphi (\lambda ) | \leq C ( 1 + | \lambda | ) , \ \ 
\Im \lambda \geq 0 . \]}

We next show that
\begin{equation}
\label{eq:FB0}  \tr \mathcal R( \lambda ) =  \tr B ( \lambda ) \,, 
\end{equation}
where $ B ( \lambda ) $ is given by \eqref{eq:Bla}.

Observe by cyclicity of the trace (cf. \cite[Theorem B.4.9]{dyatlov2019mathematical})
\begin{align}
    \tr (B(\lambda )) &= -2\lambda \tr (R_V (\lambda)V R_0(\lambda)) \\
    &= -2\lambda \tr(R_0 (\lambda ) V R_0(\lambda)) - 2 \lambda \tr((R_V(\lambda) - R_0(\lambda) V R_0(\lambda))\\
    &= -2\lambda \tr(V^{1/2} R_0^2(\lambda)|V|^{1/2} ) - 2 \lambda \tr((R_V(\lambda) - R_0(\lambda) V R_0(\lambda))\\
    &= -2\lambda \tr(V^{1/2} R_0^2(\lambda)|V|^{1/2} )  + 2 \lambda \tr(R_0 (\lambda )V R_V (\lambda) V R_0(\lambda)) \label{eq:1589}
\end{align}
where in the last equality, we use the resolvent identity $(R_0 -R_V ) V R_0 = R_0 V R_V $.

To compute $\tr \mathcal{R}(\lambda)$, we first use the fact that $ R_0 ( \lambda ) $ is
bounded
on $ L^2 $ for $ \Im \lambda > 0 $ and hence
\[ \partial_\lambda  ( \mathbf R_V ( \lambda )  ) = 2 \lambda V^{\frac12} R_0  (
\lambda)^2 |V|^{\frac12}   \,. \]
Therefore $\tr \mathcal R (\lambda) $ is equal to
\begin{align}
    &-2 \lambda \tr (V^{1/2} R_0 (\lambda )^2 |V|^{1/2} (1 + \textbf{R}_V(\lambda) ^{-1})\\
    &=-2 \lambda \tr (V^{1/2} R_0(\lambda)^2 |V|^{1/2}) - 2 \lambda \tr (V^{1/2} R_0^2(\lambda ) |V|^{1/2}((1 + \textbf{R}_V(\lambda ))^{-1} -1)). \label{eq:1601}
\end{align}
The second term on the right-hand side can be written
\begin{align}
    &-2 \lambda \tr (V^{1/2} R_0^2(\lambda ) |V|^{1/2}(1 + \textbf{R}_V(\lambda ))^{-1}(1 - (1 +\textbf{R}_V(\lambda))) ) \\
    &\quad =2 \lambda \tr (V^{1/2} R_0^2(\lambda ) |V|^{1/2}(1 + \textbf{R}_V(\lambda ))^{-1}(\textbf{R}_V(\lambda)) )\\
    &\quad = 2\lambda \tr (V^{1/2} R_0 ^2 (\lambda) V R_V (\lambda ) |V|^{1/2})=2 \lambda \tr (R_0(\lambda ) V R_V (\lambda )VR_0(\lambda)) \label{eq:1607}
\end{align}
where in the penultimate equality we use that $$(1 + \textbf{R}_V(\lambda ))^{-1}\textbf{R}_V(\lambda ) = V^{1/2} R_V (\lambda ) |V|^{1/2}$$ and in the final equality we use cyclicity of the trace.

Comparing \eqref{eq:1589} with \eqref{eq:1601} and \eqref{eq:1607} gives \eqref{eq:FB0}. 
\end{proof}

\Step We lastly compute the trace of $f(P_V ) - f(P_0)$.

From \Cref{lem:1560} we get for $\e > 0$
\begin{align}
    \tr\left(f(P_V) - f(P_0)\right) &= \frac{1}{2\pi}\sum_\pm \pm \int_{\set{\pm \Im \lambda > 0}\cap\set{|\lambda|>\e}} \dbar _\lambda g(\lambda ) \tr \mathcal R(\pm \lambda )\dd m(\lambda ) \\
    &+\frac{1}{2\pi}\sum_\pm \pm \int_{\set{\pm \Im \lambda > 0}\cap\set{|\lambda|\le\e}} \dbar _\lambda g(\lambda ) \tr \mathcal R(\pm \lambda )\dd m(\lambda ).
\end{align}
The second term is $\mathcal O(\e^\infty)$ using that $\Tr \mathcal R (\lambda ) = \tr B(\lambda)$ \eqref{eq:FB0}, the trace norm bound on $B$ \eqref{eq:Blah}, and $\dbar$ bounds on $g$ \eqref{eq:faag}.
We next apply the Cauchy--Green formula (cf. \cite[(D.1.1)]{dyatlov2019mathematical}) on the first term to get
\begin{align}
&\frac{1}{2\pi}\sum_\pm \pm \int_{\set{\pm \Im \lambda > 0}\cap\set{|\lambda|>\e}} \dbar _\lambda g(\lambda ) \tr \mathcal R(\pm \lambda )\dd m(\lambda )\\
& \qquad \qquad \qquad \qquad =\frac{1}{4\pi i } \sum_\pm \ \int_{\gamma _\pm (\e)} g(\lambda ) \Tr  \mathcal R (\pm \lambda ) \dd \lambda
\end{align}
where $\gamma_\pm(\e)$ is the contour going from $-\infty$ to $-\e$, then along the semicircle $|z| = \e$ to $+\e$ such that $\pm\Im {\gamma_\pm (\e)} \ge 0$, then from $+\e $ to $\infty$. 
Note that
\begin{align}
    \int_{\gamma_\pm (\e)} g(\lambda ) \Tr \mathcal R(\pm \lambda ) \dd \lambda &=  \int_{\gamma_\pm (\e)\cap \R} g(\lambda ) \Tr \mathcal R(\pm \lambda ) \dd \lambda  \\
    &+ \int_{\gamma_\pm (\e)\setminus \R} g(\lambda ) \Tr \mathcal R(\pm \lambda ) \dd \lambda.\label{eq:1942}
\end{align}
The sum over $\pm$ of the second term can be estimated using \eqref{eq:trF0}
\begin{align}
   \sum _\pm  \int_{\gamma_\pm (\e)\setminus \R} g(\lambda ) \Tr \mathcal R(\pm \lambda ) \dd \lambda &= (m_R(0) - 1) f(0) \oint_{|\lambda| = 1} \frac{\dd \lambda}{\lambda } + \mathcal{O}(\e)\\
   &= 2\pi i  (m_R(0) - 1)f(0) + \mathcal{O}(\e).
\end{align}
The sum over $\pm$ of the first term of \eqref{eq:1942} can be written, using \eqref{eq:Fsc},
\begin{align}
    \int_{-\infty}^\infty f(\lambda ^2 ) \Tr \p_\lambda S(\lambda ) S(\lambda)^{-1} \dd \lambda- \int_{-\e}^{\e} f(\lambda ^2 ) \Tr \p_\lambda S(\lambda ) S(\lambda)^{-1} \dd \lambda.
\end{align}
The second term is $\mathcal{O}(\e)$.
Sending $\e\to 0$, and using that $\Tr \p_\lambda S(\lambda)S(\lambda)^{-1}$ is even, gives the theorem.

\setcounter{step}{0}
\end{proof} 
\textbf{Data availability statement}. The numerics used to generate \Cref{fig:theonlyfigure} can be found at \url{https://github.com/ioltman/trace_SdS}.
\\

\textbf{Conflict of interest statement}. The authors declare that there is no conflict of interest.
\printbibliography

\end{document}